\documentclass[11pt]{article}

\usepackage[colorinlistoftodos]{todonotes}
\definecolor{darkred}{rgb}{0.7,0,0} 
\newcommand{\defn}[1]{{\color{darkred}\emph{#1}}} 

\usepackage{latexsym}
\usepackage{amssymb}
\usepackage{amsthm}
\usepackage{amscd}
\usepackage{amsmath}
\usepackage{mathrsfs}
\usepackage{graphicx}
\usepackage{shuffle}
\usepackage{hyperref}
\usepackage{mathtools}
\usepackage[bbgreekl]{mathbbol}
\usepackage{mathdots}
\usepackage{ytableau}
\usepackage{tikz-cd}

\usepackage{stmaryrd}
\usepackage{young}

\usepackage[all]{xy}
\input xy \xyoption{frame}
\xyoption{dvips}

\usepackage[colorinlistoftodos]{todonotes}
\usetikzlibrary{chains,scopes}
\usetikzlibrary{shapes.geometric,positioning}

\DeclareSymbolFontAlphabet{\mathbb}{AMSb}
\DeclareSymbolFontAlphabet{\mathbbl}{bbold}

\newcommand{\MRPP}{\ensuremath\mathrm{ShRPP}}

\numberwithin{equation}{section}
\allowdisplaybreaks[1]
\UseCrayolaColors

\theoremstyle{definition}
\newtheorem* {theorem*}{Theorem}
\newtheorem* {conjecture*}{Conjecture}
\newtheorem{theorem}{Theorem}[section]

\newtheorem{problem}[theorem]{Problem}
\theoremstyle{definition}

\newtheorem* {example*}{Example}

\newtheorem{lemma}[theorem]{Lemma}
\theoremstyle{definition}
\newtheorem{definition}[theorem]{Definition}
\theoremstyle{definition}

\newtheorem{conjecture}[theorem]{Conjecture}
\newtheorem{proposition}[theorem]{Proposition}
\newtheorem{corollary}[theorem]{Corollary}

\newtheorem*{remark*}{Remark}
\theoremstyle{definition}
\newtheorem{remark}[theorem]{Remark}
\theoremstyle{definition}
\newtheorem {example}[theorem]{Example}
\theoremstyle{definition}

\theoremstyle{definition}

\theoremstyle{definition}

\xyoption{dvips}

\def\({\left(}
\def\){\right)}

\newcommand{\QQ}{\mathbb{Q}}

\newcommand{\cI}{\mathcal{I}}

\def\Hom{\mathrm{Hom}}
\def\End{\mathrm{End}}

\def\ZZ{\mathbb{Z}}

\def\spanning{\textnormal{-span}}

\def\fk{\mathfrak}

\def\barr{\begin{array}}
\def\earr{\end{array}}
\def\ba{\begin{aligned}}
\def\ea{\end{aligned}}
\def\be{\begin{equation}}
\def\ee{\end{equation}}

\def\quand{\quad\text{and}\quad}

\def\M{\mathcal{M}}

\def\ds{\displaystyle}

\def\id{\mathrm{id}}

\def\ben{\begin{enumerate}}
\def\een{\end{enumerate}}

\def\Des{\mathrm{Des}}

\def\c{{\tt c}}
\def\t{{\tt t}}
\def\r{{\tt r}}

\def\arcstart{\ \xy<0cm,-.15cm>\xymatrix@R=.1cm@C=.3cm }
\newcommand{\arcstartc}[1]{\ \xy<0cm,-.15cm>\xymatrix@R=.1cm@C=#1cm}

\def\M{\mathbb{M}}

\def\st{\mathfrak{st}}

\def\antipode{{\tt S}}

\def\htimes{\mathbin{\hat\otimes}}

\def\zetaq{\zeta_{\textsf{Q}}}

\def\bK{\bar{K}}

 \newcommand{\Sym}{\textsf{Sym}}
\newcommand{\NSym}{\textsf{NSym}}
\newcommand{\MPR}{\textsf{MR}}
\newcommand{\QSym}{\textsf{QSym}}

\newcommand{\Peak}{\textsf{Peak}}
\newcommand{\MPeakQ}{\mathfrak{M}\Peak_Q}
\newcommand{\MPeakP}{\mathfrak{M}\Peak_P}

\newcommand{\MSymP}{\mathfrak{M}\Sym_P}
\newcommand{\mSymQ}{\mathfrak{m}\Sym_Q}

\newcommand{\MSymQ}{\mathfrak{M}\Sym_Q}
\newcommand{\mSymP}{\mathfrak{m}\Sym_P}

\newcommand{\MSymPQ}{\mathit{g\Gamma}}
\newcommand{\mSymPQ}{\mathit{G\Gamma}}

\newcommand{\QSymQ}{\Pi\Sym}
\newcommand{\QSymP}{\bar{\Pi}\Sym}
\newcommand{\mQSymQ}{\mathfrak{m}\QSymQ}
\newcommand{\mQSymP}{\mathfrak{m}\QSymP}

\newcommand{\MSym}{\mathfrak{M}\textsf{Sym}}
\newcommand{\MNSym}{\mathfrak{M}\textsf{NSym}}
\newcommand{\MMPR}{\mathfrak{M}\textsf{MR}}

\newcommand{\mMPR}{\mathfrak{m}\textsf{MR}}

\def\columns{\mathsf{cols}}

\def\PWords{\textsf{PackedWords}}
\def\WQSym{\textsf{WQSym}}
\def\mWQSym{\mathfrak{m}\textsf{WQSym}}
\def\SetComp{\textsf{SetComp}}
\def\mWQSym{\m\textsf{WQSym}}
\def\MWQSym{\textsf{WQSym}}
\def\PSet{\textsf{Set}(\PP)}
\def\MSet{\textsf{Set}(\MM)}

\def\ttheta{\Theta^{(\beta)}}

\def\tR{R^{(\beta)}}

\def\oK{\bar{K}}

\newcommand{\mSym}{\mathfrak{m}\textsf{Sym}}
\newcommand{\mQSym}{\mathfrak{m}\textsf{QSym}}

\def\Sm{S^{\fk m}}
\def\SM{S^{\fk M}}
\def\barSM{{\bar S}^{\fk M}}

\def\XX{\mathbb{X}}

\def\m{\mathfrak{m}}
\def\M{\mathfrak{M}}

\def\wtmrpp{\mathrm{wt}_{\mathrm{RPP}}}

\def\SD{\textsf{SD}}

\def\equivM{\mathbin{<_\M}}

\def\cJM{\cI_\M}

\def\PeakSet{\mathrm{Peak}}

\def\opeak{\pi\hspace{-0.1mm}q}
\def\tpeak{\pi\hspace{-0.1mm}p}

\def\tpeakR{\tpeak^{(\beta)}}
\def\opeakR{\opeak^{(\beta)}}

\def\GQ{G\hspace{-0.2mm}Q}
\def\GP{G\hspace{-0.2mm}P}
\def\JQ{J\hspace{-0.2mm}Q}
\def\JP{J\hspace{-0.2mm}P}

\def\bGQ{\GQ^{(\beta)}}
\def\bGP{\GP^{(\beta)}}

\def\bJQ{\JQ^{(\beta)}}
\def\bJP{\JP^{(\beta)}}

\def\bXi{\Xi^{(\beta)}}

\def\gq{g\hspace{-0.1mm}q}
\def\gp{g\hspace{-0.1mm}p}

\def\bgq{\gq^{(\beta)}}
\def\bgp{\gp^{(\beta)}}

\def\jq{j\hspace{-0.1mm}q}
\def\jp{j\hspace{-0.1mm}p}

\def\bjq{\jq^{(\beta)}}
\def\bjp{\jp^{(\beta)}}

\def\JQ{J\hspace{-0.2mm}Q}
\def\JP{J\hspace{-0.2mm}P}

\def\ss{/\hspace{-1mm}/}

\def\MM{\mathbb{M}}
\def\PP{\ZZ_{>0}}
\def\NN{\ZZ_{\geq 0}}

\newcommand{\cC}{\alpha_{\textsf{des}}}
\newcommand{\cD}{\alpha_{\textsf{peak}}}

\def\RC{\mathsf{RC}}
\def\ShSetTab{\mathsf{ShSVT}}
\def\ShMSetTab{\mathsf{ShMVT}}

\def\ShBTQ{\mathsf{ShBT}_Q}
\def\ShBTP{\mathsf{ShBT}_P}

\def\MM{\frac{1}{2}\PP}
\def\cons{o}

\def\mG{\mathbf{G}^+}
\def\mGP{\mathbf{GP}^+}
\def\mGQ{\mathbf{GQ}^+}

\def\mJP{\widehat{\mathbf{JP}}{}^+}
\def\mJQ{\widehat{\mathbf{JQ}}{}^+}

\def\mg{\mathbf{g}^+}
\def\mgp{\mathbf{gp}^+}
\def\mgq{\mathbf{gq}^+}

\usepackage{fullpage}
\usepackage[backend=biber]{biblatex}
\begin{document}
\title{Shifted combinatorial Hopf algebras from $K$-theory}
\author{
Eric Marberg \\ Department of Mathematics \\  Hong Kong University of Science and Technology \\ {\tt emarberg@ust.hk}
}

\date{}

\maketitle

\begin{abstract}
In prior joint work with Lewis, we developed a theory of enriched set-valued $P$-partitions to construct a $K$-theoretic generalization of the Hopf algebra of peak quasisymmetric functions. Here, we situate this object in a diagram of six Hopf algebras, providing a shifted version of the diagram of $K$-theoretic combinatorial Hopf algebras studied by Lam and Pylyavskyy. This allows us to describe new $K$-theoretic analogues of the classical peak algebra. We also study the Hopf algebras generated by Ikeda and Naruse's $K$-theoretic Schur $P$- and $Q$-functions, as well as their duals. Along the way, we derive several product, coproduct, and antipode formulas and outline a number of open problems and conjectures.
\end{abstract}


\section{Introduction}

There is a classical diagram of Hopf algebras
\be\label{classical-diagram}
\begin{tikzcd}
\Sym \arrow[d, dash] & \NSym \arrow[l, twoheadrightarrow] \arrow[d, dash]  \arrow[r, hook] & \MPR \arrow[d, dash] \\
\Sym \arrow[r,hook] & \QSym  & \MPR  \arrow[l,twoheadrightarrow]
\end{tikzcd}
\ee
in which  the vertical lines are dualities, the $\hookrightarrow$ arrows are inclusions, and 
the $\twoheadrightarrow$ arrows are their adjoints.
The self-dual object $\Sym$ on the left is the familiar Hopf algebra 
of bounded degree symmetric functions \cite[\S2]{GrinbergReiner}, which has an orthonormal basis given by the  Schur functions $s_\lambda$.
The self-dual object $\MPR$ on the right is the \defn{Malvenuto-Reutenauer Hopf algebra}
of permutations from \cite{AguiarSottile,Malvenuto}.
In the middle, we have the dual pair of \defn{quasisymmetric functions} $\QSym$
and noncommutative symmetric functions $\NSym$, as described in \cite[\S5]{GrinbergReiner}.

In \cite{LamPyl}, Lam and Pylyavskyy  study a ``$K$-theoretic'' generalization of \eqref{classical-diagram} 
given by 
\be\label{k-diagram}
\begin{tikzcd}
\MSym \arrow[d, dash] & \MNSym \arrow[l, twoheadrightarrow] \arrow[d, dash]  \arrow[r, hook] & \MMPR \arrow[d, dash] \\
\mSym \arrow[r,hook] & \mQSym  & \mMPR  \arrow[l,twoheadrightarrow]
\end{tikzcd}.
\ee
The objects here are modules over $\ZZ[\beta]$ rather than $\ZZ$, where $\beta$ is a formal parameter. Setting $\beta=0$ turns \eqref{k-diagram} into \eqref{classical-diagram}.
The objects $\mSym$ and $\mQSym$ consist of the symmetric and quasisymmetric functions over $\ZZ[\beta]$
of unbounded degree. 
Their duals $\MSym$ and $\MNSym$ are isomorphic to $\Sym$ and $\NSym$ but with scalars extended to $\ZZ[\beta]$.
The objects $\mMPR$ and $\MMPR$, finally, are two different generalizations of 
the Malvenuto-Reutenauer Hopf algebra $
\MPR$.

Besides the Schur functions $\{s_\lambda\}$, the Hopf algebras $\mSym$ and $\MSym$
have another pair of dual bases provided by the \defn{stable Grothendieck polynomials}
$G^{(\beta)}_\lambda$ and the \defn{dual stable Grothendieck polynomials} $g^{(\beta)}_\lambda$.
These power series are relevant to $K$-theory since the $G^{(\beta)}_\lambda$ functions are symmetric limits of connective $K$-theory classes of structure sheaves of Schubert varieties; see \cite{Buch2002,FominKirillov94}.

The goal of this article is to investigate two shifted analogues of \eqref{k-diagram}.
To motivate this, let us first discuss the shifted versions of  \eqref{classical-diagram}.
On one hand, we have a diagram
\be\label{shifted-diagram}
\begin{tikzcd}
\Sym_P \arrow[d, dash] & \Peak_P \arrow[l, twoheadrightarrow] \arrow[d, dash]  \arrow[r, hook] & \MPR \arrow[d, dash] \\
\Sym_Q \arrow[r,hook] & \QSymQ  & \MPR  \arrow[l,twoheadrightarrow]
\end{tikzcd}
\ee
in which the vertical lines are again dualities, the $\hookrightarrow$ arrows are inclusions, and 
the $\twoheadrightarrow$ arrows are their adjoints.
Here  $\Sym_P$ and $\Sym_Q$ are the subalgebras of $\Sym$
spanned by the \defn{Schur $P$-functions} $P_\lambda$ and \defn{Schur $Q$-functions} $Q_\lambda$,
which are indexed by all \defn{strict} integer partitions $\lambda$.
These subalgebras are dual Hopf algebras relative to the bilinear form with $[ P_\lambda, Q_\mu] = \delta_{\lambda\mu}$,
which is different from the usual Hall inner product on $\Sym$;
see \cite[Appendix A]{Stembridge1997a}.
In the middle column, $\QSymQ$ is the Hopf algebra of \defn{peak quasisymmetric functions}
$\QSymQ$ from \cite{Stembridge1997a} while  $\Peak_P$ is the \defn{peak algebra} from \cite{Nyman,Schocker}.

There is another version of \eqref{shifted-diagram} 
in which the roles of $\Sym_P$ and $\Sym_Q$ are interchanged:
\be\label{shifted-diagram2}
\begin{tikzcd}
\Sym_Q \arrow[d, dash] & \Peak_Q \arrow[l, twoheadrightarrow] \arrow[d, dash]  \arrow[r, hook] & \MPR \arrow[d, dash] \\
\Sym_P \arrow[r,hook] & \QSymP  & \MPR  \arrow[l,twoheadrightarrow]
\end{tikzcd}.
\ee
Here $\QSymP$ is a slightly larger version of $\QSymQ$ (namely,  the intersection of $\QSymQ_\QQ := \QQ \otimes_\ZZ \QSymQ$ with $\QSym$ \cite[\S3]{Stembridge1997a})
while its dual $\Peak_Q$ is a Hopf subalgebra of $\Peak_P$.
The diagrams \eqref{shifted-diagram} and \eqref{shifted-diagram2} coincide if we work over the scalar ring $\QQ$ rather than $\ZZ$.

Work of Ikeda and Naruse  \cite{IkedaNaruse} identifies
$K$-theoretic versions $\bGP_\lambda$ and $\bGQ_\lambda$ of the classical Schur $P$- and $Q$-functions.
Whereas $P_\lambda$ and $Q_\lambda$
are generating functions for \defn{(semistandard) shifted tableaux},
$\bGP_\lambda$ and $\bGQ_\lambda$  are generating functions for 
\defn{(semistandard) shifted set-valued tableaux} \cite[Thm.~9.1]{IkedaNaruse}.
These symmetric functions represent the structure sheaves of Schubert varieties in the connective $K$-theory ring of the maximal isotropic Grassmannians of orthogonal and symplectic types \cite[Cor.~8.1]{IkedaNaruse}.

Later results of Nakagawa and Naruse \cite{NN2018} construct
two additional families of ``dual'' $K$-theoretic Schur $P$- and $Q$-functions $\bgp_\lambda$ and $\bgq_\lambda$.
As we will explain in Section~\ref{main-sect}, these power series are $\ZZ[\beta]$-bases 
for two Hopf subalgebras $\MSymP$ and $\MSymQ$ of $\MSym$, 
whose respective duals
 $\mSymQ$ and $\mSymP$
  are 
the completions of the algebras $\ZZ[\beta]\spanning\{\bGQ_\lambda\}$ and $\ZZ[\beta]\spanning\{\bGP_\lambda\}$.
These four objects fit into the pair of
diagrams
\be\label{shk-diagram}
\begin{tikzcd}
\MSymP   \arrow[d, dash] &  \MPeakP   \arrow[l, twoheadrightarrow]   \arrow[d, dash]  \arrow[r, hook] & \MMPR \arrow[d, dash] \\
\mSymQ   \arrow[r, hook] &  \mQSymQ   & \mMPR  \arrow[l, twoheadrightarrow]
\end{tikzcd}
\qquad
\begin{tikzcd}
\MSymQ   \arrow[d, dash] &  \MPeakQ   \arrow[l, twoheadrightarrow]   \arrow[d, dash]  \arrow[r, hook] & \MMPR \arrow[d, dash] \\
\mSymP   \arrow[r, hook] &  \mQSymP   & \mMPR  \arrow[l, twoheadrightarrow]
\end{tikzcd}
\ee
which specialize  to \eqref{shifted-diagram} and \eqref{shifted-diagram2} when $\beta=0$,
and which coincide if the scalar ring $\ZZ[\beta]$ is extended to $\QQ[\beta]$.
These diagrams are the primary subject of this article. 
Our main results, building off related work in \cite{ChiuMarberg,LamPyl,LM2022,LM2019}, 
will provide 
distinguished bases for all of the objects shown here, 
explicitly identify the pairings that give the dualities indicated by the vertical lines,
and  
describe the remaining inclusions and their adjoint surjections.

We can summarize our new theorems and outline the rest of this paper as follows.
One way to motivate \eqref{k-diagram}  is through the perspective of \defn{combinatorial Hopf algebras}
as defined in \cite{ABS}. However, some care must be taken to make this rigorous 
as the objects in the bottom row of \eqref{k-diagram} are certain completions of Hopf algebras
rather than actual Hopf algebras.
Section~\ref{prelim-sect} provides a brief survey of the technical background needed to address these issues.

  Section~\ref{k-sect} reviews the construction of the objects and morphisms
in \eqref{k-diagram}. What we present is a very mild generalization of what is studied
by Lam and Pylyavskyy, and involves a parameter $\beta$ that is implicitly set to $\beta=1$ in \cite{LamPyl}.

 The algebras $\mQSymP\supset \mQSymQ$ are spanned by $K$-theoretic generalizations of 
peak quasisymmetric functions studied previously in \cite{LM2019}.
In Section~\ref{multipeak-sect}
we review the definition of these algebras, and prove that $\mQSymQ$ arises of the image of
a canonical morphism of combinatorial Hopf algebras $\mMPR \to \mQSym$; see Theorem~\ref{peak-tl-thm}.

  In Section~\ref{MPeak-sect} we construct the duals of 
$\mQSymP$ and $\mQSymQ$ as   explicit Hopf subalgebras  $\MPeakQ\subset \MNSym$ and $\MPeakP\subset \MNSym$.
This gives two $K$-theoretic generalizations of the classical peak algebra.
Neither appears to have been considered 
in previous literature.
We also derive (co)product formulas for the distinguished bases of 
 $\MPeakQ$ and $\MPeakP$,
 and we identify the adjoint maps $\mMPR \to \mQSymQ$ and $\mMPR \to \mQSymP$ in \eqref{shk-diagram}.
 
  Sections~\ref{shifted-sym-sect} and \ref{pq-sect2}
 discuss the Hopf algebras $\mSymP\supset \mSymQ$
 and their duals $\MSymQ\subset \MSymP$. We 
 prove an identity relating the pairings  $\MSymQ\times \mSymP \to \ZZ[\beta]$
and $\MSymP\times \mSymQ \to \ZZ[\beta]$ to a surjective morphism $\ttheta : \mQSym \to \mQSymQ$;
see Theorem~\ref{ttheta-thm}.
We also identify the adjoint maps $\MPeakP \to \MSymP$
and $\MPeakQ \to \MSymQ$ in \eqref{shk-diagram}; see Theorem~\ref{last-adj-thm}.
These results rely on conjectures from \cite{LM2019,NN2018} proved in \cite{ChiuMarberg,LM2022}.

  Section~\ref{antipode-sect} discusses antipode formulas for the objects in \eqref{shk-diagram}. Building off recent work in \cite{LM2022}, we show that the respective (finite) $\ZZ[\beta]$-linear
spans of all $\bGP$- and $\bGQ$-functions are   sub-bialgebras of $\mSym$ that are not Hopf algebras; see Proposition~\ref{bialgebras-thm}.

  Finally, Section~\ref{open-sect} provides a survey of related open problems and positivity conjectures.
Computer calculations indicate that the coefficients appearing in many different expansions of the distinguished 
bases for the objects in \eqref{shk-diagram} are always positive. In several special cases, it is an open problem to find combinatorial
interpretations of these numbers.

 \subsection*{Acknowledgements}

This work was partially supported by Hong Kong RGC grant GRF 16306120.
We thank Joel Lewis for many helpful discussions
and for his hospitality during several productive visits to George Washington University.

\section{Preliminaries}\label{prelim-sect}

Throughout, we write $\ZZ$ for the set of integers
and let $[n] := \{i \in \ZZ : 0<i \leq n\}$ for $0\leq n \in \ZZ$.
This section presents some basic information about Hopf algebras, their completions,
and quasisymmetric functions.
For more background on each, see \cite{GrinbergReiner,LMW,Marberg2018}.

\subsection{Hopf algebras}

Fix an integral domain $R$ and write $\otimes = \otimes_R$ for the tensor product over this ring.
An \defn{$R$-algebra} is an $R$-module $A$ with $R$-linear
product $\nabla : A\otimes A \to A$
and unit $\iota : R\to A$ 
maps.
Dually, an \defn{$R$-coalgebra} is an $R$-module $A$ with $R$-linear 
coproduct $\Delta : A \to A\otimes A$
and 
counit $\epsilon : A \to R$ maps.
 The (co)product and (co)unit maps must satisfy several associativity axioms;
see \cite[\S1]{GrinbergReiner} for the complete definitions.

An $R$-module $A$ that is both an $R$-algebra and an $R$-coalgebra is an \defn{$R$-bialgebra}
if the coproduct and counit maps are algebra morphisms (equivalently, the product and unit are coalgebra morphisms).

Suppose $A$ is an $R$-bialgebra with structure maps $\nabla$, $\iota$, $\Delta$, and $\epsilon$. 
Let $\End(A)$ denote the set of $R$-linear maps $A \to A$. This set is an $R$-algebra for the 
product 
$
f \ast g := \nabla \circ (f\otimes g) \circ \Delta
$. 
The unit of this \defn{convolution algebra}
is the composition $ \iota\circ \epsilon$ of the unit and counit of $A$.
The bialgebra $A$ is a \defn{Hopf algebra} if  $\id : A \to A$ has a
(necessarily unique) two-sided inverse 
$\antipode : A \to A$ in the convolution algebra $\End(A)$. When it exists, we call $\antipode$ the \defn{antipode} of $A$.

\subsection{Completions}\label{completions-sect}

Many of the objects in the diagrams \eqref{k-diagram} and \eqref{shk-diagram} 
are rings of formal power series of unbounded degree that are ``too large'' to belong to the category of free modules. To formally define algebraic structures on these objects, we need to work in the following setting.

Let $A$ and $B$ be $R$-modules
with an $R$-bilinear form $\langle\cdot,\cdot\rangle : A \times B \to R$.
Assume that $A$ is free and the form
is \defn{nondegenerate} in the sense that
$b\mapsto \langle \cdot,b\rangle$ is a bijection $B\to \Hom_R(A,R)$.
Fix a basis $\{a_i\}_{i \in I}$ for $A$.
For each $i \in I$, there exists a unique element $b_i \in B$ with
$\langle a_i, b_j \rangle = \delta_{ij}$ for all $i,j \in I$,
and we can identify an arbitrary element $b \in B$
with the formal linear combination $\sum_{i \in I} \langle a_i ,b\rangle b_i$.
We refer to $\{b_i\}_{i \in I}$ as a \defn{pseudobasis} for $B$.

We give $R$ with the discrete topology.
Then the \defn{linearly compact topology} \cite[\S I.2]{Dieudonne} on $B$ is the coarsest topology 
in which the maps $\langle a_i, \cdot \rangle : B \to R$ are all continuous.
This topology depends on $\langle\cdot,\cdot\rangle$ but not on the choice of basis for $A$,
and is discrete if $A$ has finite rank.

\begin{definition}
A \defn{linearly compact $R$-module} is an $R$-module $B$ equipped with a nondegenerate
bilinear form $A \times B \to R$ for some free $R$-module $A$, given the linearly compact topology;
in this case  $B$ is the \defn{dual} of $A$.
Morphisms between such modules are continuous $R$-linear maps.
\end{definition}

We will often abbreviate by writing ``LC-'' in place of ``linearly compact.''
Suppose $A$ is a free $R$-module with basis $S$.
We refer to the
$R$-module $B$ of arbitrary $R$-linear combinations of $S$,
 equipped with the nondegenerate bilinear form
$ A \times B \to R$ making $S$ orthonormal,
as the 
\defn{completion} of $A$ with respect to $S$.
This linearly compact $R$-module has $S$ as a pseudobasis.

Let $B$ and $B'$ be linearly compact $R$-modules dual to free $R$-modules $A$ and $A'$.
We reuse $\langle\cdot,\cdot\rangle$ for both of the 
associated nondegenerate forms. 
Every linear map $\phi : A' \to A$ has a unique adjoint $\psi : B\to B'$ such that 
$\langle \phi(a), b\rangle = \langle a,\psi(b)\rangle$ for all $a \in A'$ and $b \in B$.
A linear map $B \to B'$ is \defn{continuous} if and only if it arises as the adjoint 
of some linear map $A' \to A$.

\begin{definition}
The \defn{completed tensor product} of $B$ and $B'$ 
is the $R$-module 
$B \htimes B' := \Hom_R(A\otimes A',R),$
given the LC-topology from the tautological pairing 
$(A\otimes A') \times \Hom_R(A\otimes A',R) \to R$.
\end{definition}

If $\{b_i\}_{i \in I}$ and $\{b'_j \}_{j \in J}$ are pseudobases for $B$ and $B'$,
then we can realize $B\htimes B'$ concretely as the linearly compact $R$-module
with the set of tensors $\{ b_i \otimes b_j'\}_{(i,j) \in I \times J}$ as a pseudobasis.

 \begin{example}\label{lc-ex0}
Let $A = R[x]$ and $B=R\llbracket x \rrbracket$.
Define $\langle\cdot,\cdot\rangle : A \times B \to R$ to be the nondegenerate $R$-bilinear form 
$\left\langle \sum_{n\geq 0} a_n x^n, \sum_{n \geq 0} b_n x^n\right\rangle := \sum_{n\geq 0} a_n b_n.$
Then the set $\{x^n\}_{n\geq 0}$ is a basis for $A$ and a pseudobasis for $B$,
and we have $R\llbracket x \rrbracket \otimes R\llbracket y \rrbracket \neq R\llbracket x \rrbracket\htimes R\llbracket y \rrbracket \cong R\llbracket x,y \rrbracket$
\end{example}

\begin{definition}
Suppose $\nabla : B \htimes B \to B$ and $\iota : B \to R$ are continuous linear maps
which are the adjoints of linear maps $\epsilon : R \to A$ and $\Delta : A \to A \otimes A$.
We say that $(B,\nabla,\iota)$ is an \defn{LC-algebra}
if $(A,\Delta,\epsilon)$ is an $R$-coalgebra.
Similarly, we say that 
$\Delta : B \to B\htimes B$ and $\epsilon : B \to R$ 
make $B$ into an \defn{LC-coalgebra}
if $\Delta$ and $\epsilon$ are the adjoints of the product and unit 
maps of an $R$-algebra on $A$.
 
\end{definition}

We define
\defn{LC-bialgebras} and \defn{LC-Hopf algebras} analogously.
If $B$ is an LC-Hopf algebra then its \defn{antipode} is the adjoint of the antipode of 
the Hopf algebra $A$.
In each case we say that the (co, bi, Hopf) algebra structures on $A$ and $B$ are duals of each other.

More generally, a Hopf algebra $H$ is \defn{dual} to an LC-Hopf algebra $\hat H$ 
via some nondegenerate bilinear form $\langle\cdot,\cdot\rangle : H \times \hat H \to R$ that is continuous in the second coordinate
if one always has $\langle \iota(a), b\rangle = a\cdot \epsilon(b)$, $\langle a, \iota(b)\rangle = \epsilon(a)\cdot b$,
$\langle \nabla(a \otimes b), c\rangle = \langle a\otimes b,\Delta(c)\rangle$,
and 
$\langle a, \nabla(b \otimes c)\rangle = \langle \Delta(a), b\otimes c \rangle$.

 \begin{example} 
Again let $A = R[x]$ and $B=R\llbracket x \rrbracket$. 
Write $\nabla : B \htimes B \to B$ for the usual product map,
define $\iota : R \to B$ to be the natural inclusion, 
and
let $\epsilon : B \to R$ be the ring homomorphism that sets $x=0$.
For each   $\beta \in R$, there is a continuous algebra homomorphism 
$\Delta_\beta : B \to B \htimes B$ with $\Delta_\beta(x) = x\otimes 1 + 1\otimes x+ \beta x \otimes x$,
and $B$ is an LC-bialgebra relative to $\nabla$, $\iota$, $\Delta_\beta$, and $\epsilon$.

There is a unique bialgebra structure on $A$ that is dual to the one on $B$ via the form in Example~\ref{lc-ex0}.
This structure has unit, counit, and coproduct given by appropriate restrictions of $\iota$, $\epsilon$, and $\Delta_0 := \Delta_\beta|_{\beta=0}$,
while its product  has a more complicated formula; see \cite[Ex. 2.4]{LM2019}.
One can show that the dual bialgebras $A$ and $B$ are dual Hopf algebras:
the antipode of $A$ is the linear map with $\antipode_A(x^n)  = (-1)^ n x(x+\beta)^{n-1}$,
and the antipode of $B$ is the continuous linear map with $\antipode_B(x^n)  =\(\frac{-x}{1+\beta x}\)^n$.
Notice that we can restrict $\nabla$ and $\Delta_\beta$ to define a second bialgebra structure on $A$, 
but this will not be a Hopf algebra unless $\beta=0$ as $R$ is an integral domain. 

\end{example}

One can  reformulate the above definitions using commutative diagrams; see 
\cite{Marberg2018}.
Linearly compact (co, bi, Hopf) algebras form a category in which morphisms are continuous linear maps 
commuting with (co)products and (co)units.
The completed tensor product of two linearly compact (co, bi, Hopf) algebras is
itself a linearly compact (co, bi, Hopf) algebra.

\subsection{Quasisymmetric functions}

Let $\beta$, $x_1$, $x_2$, \dots be commuting indeterminates.
From this point on, most of our modules will be defined over the ring $R=\ZZ[\beta]$,
and we write $\otimes $ and $\htimes$ for the corresponding tensor products.
A power series $f \in \ZZ[\beta]\llbracket x_1,x_2,\dots\rrbracket $ is \defn{quasisymmetric} if
for any choice of $a_1,a_2,\dots,a_k \in \PP$,
the coefficients of $x_1^{a_1}x_2^{a_2}\cdots x_{k}^{a_k}$ and $x_{i_1}^{a_1}x_{i_2}^{a_2}\cdots x_{i_k}^{a_k}$
in $f$ are equal for all $i_1<i_2<\dots<i_k$.

\begin{definition}
Let $\mQSym$ be the $\ZZ[\beta]$-module of
quasisymmetric power series in $\ZZ[\beta]\llbracket x_1,x_2,\dots\rrbracket $.
Let $\QSym$ denote the submodule of power series in $\mQSym$ of bounded degree.
\end{definition}

A \defn{composition} $\alpha$ is a finite sequence of positive integers.
If the parts of $\alpha$ have sum $N \geq 0$, then we write $\alpha\vDash N$ or $|\alpha|=N$.
The sequence $\alpha$ is a \defn{partition} if it is weakly decreasing, which we indicate by writing $\alpha\vdash N$ instead of $\alpha \vDash N$.

The set $\QSym$ is a graded ring that is free as a $\ZZ[\beta]$-module.
One basis is provided by the
 \defn{monomial quasisymmetric functions}, which are
defined for each composition $\alpha = (\alpha_1,\alpha_2,\dots,\alpha_k)$
as the sums
$
M_\alpha := \sum_{i_1<i_2<\dots<i_k} x_{i_1}^{\alpha_1} x_{i_2}^{\alpha_2}\cdots x_{i_k}^{\alpha_k} \in \QSym
$ with 
 $M_\emptyset := 1$
when $\alpha =\emptyset$ is the empty composition.
We identify $\mQSym$ with the completion of $\QSym$ relative to this basis.

 Write $\Delta : \QSym \to \QSym \otimes \QSym$ for the $\ZZ[\beta]$-linear map with 
$
\Delta(M_\alpha) = \sum_{\alpha =\alpha'\alpha''} M_{\alpha'} \otimes M_{\alpha''}
$
for each composition $\alpha$,
where $\alpha'\alpha''$ denotes the concatenation of   $\alpha'$ and $\alpha''$.
Let $\epsilon : \QSym \to \ZZ[\beta]$ be the linear map with $M_\emptyset \mapsto 1$ and $M_\alpha\mapsto 0$
for all nonempty compositions $\alpha$.
This coproduct and counit make the algebra
 $\QSym$ into a (graded, connected) Hopf algebra \cite[\S5.1]{GrinbergReiner}.
These maps extend to continuous linear maps $\mQSym \htimes \mQSym \to\mQSym$
and
$\mQSym \to \mQSym \htimes \mQSym$ which make $\mQSym$ into an LC-Hopf algebra.
For a description of its antipode, see Section~\ref{antipode-sect}.

Suppose $H $
is an LC-bialgebra, defined over $\ZZ[\beta]$,
with product $\nabla$, coproduct $\Delta$, unit $\iota$, and counit $\epsilon$.
Let $\XX(H)$ be the set of continuous algebra morphisms $\zeta : H \to \ZZ[\beta]\llbracket t\rrbracket $
with $\zeta(\cdot )|_{t=0}=\epsilon$.

\begin{definition}
If $H$ is an LC-bialgebra (respectively, LC-Hopf algebra)
and $\zeta \in \XX(H)$, then  $(H,\zeta)$
is a \defn{combinatorial LC-bialgebra} (respectively, \defn{combinatorial LC-Hopf algebra}).
Such pairs form a category in which morphisms $(H,\zeta) \to (H',\zeta')$
are morphisms $\phi : H \to H'$ with $\zeta = \zeta'\circ \phi$.
\end{definition}

We view $\mQSym$ as a combinatorial LC-Hopf algebra with respect to
the \defn{canonical zeta function} $\zetaq : \mQSym \to \ZZ[\beta]\llbracket t\rrbracket $ 
given by $\zetaq(f) = f(t,0,0,\dots)$.
On monomial quasisymmetric functions, we have 
$\zetaq(M_\alpha)   = t^{|\alpha|}$ for $\alpha\in\{\emptyset,(1),(2),(3),\dots\}$ and $\zetaq(M_\alpha)=0$ for all other compositions $\alpha$.

For each integer $k\geq 1$ define $\Delta^{(k)} := (1 \otimes \Delta^{(k-1)})\circ \Delta = ( \Delta^{(k-1)}\otimes 1 )\circ \Delta$ where $\Delta^{(1)}:=\Delta$. 
Given
a map $\zeta \in \XX(H)$ and a nonempty composition $\alpha
=(\alpha_1,\alpha_2,\dots,\alpha_k)$,
write $\zeta_\alpha : H \to \ZZ[\beta]$ for the map
sending $h \in H$ to the coefficient of 
$t^{\alpha_1}\otimes t^{\alpha_2}\otimes \cdots \otimes t^{\alpha_k}$
in $\zeta^{\otimes k} \circ \Delta^{(k-1)}(h) \in \ZZ[\beta]\llbracket t\rrbracket ^{\htimes k}$.
When $\alpha=\emptyset$ is the empty composition,
define $\zeta_\emptyset := \epsilon$.

\begin{theorem}[{\cite[Thm.~2.8]{LM2019}}]
\label{abs-thm}
Suppose $(H,\zeta)$ is a combinatorial LC-bialgebra.
Then 
the map with the formula
$\Phi(h) = \sum_\alpha \zeta_\alpha(h) M_\alpha$
for $h \in H$,
where the sum is over all compositions $\alpha$,
is the  unique morphism 
of combinatorial LC-bialgebras
$\Phi : (H,\zeta) \to (\mQSym,\zetaq)$.
\end{theorem}

\section{$K$-theoretic Hopf algebras}\label{k-sect}

We are now prepared to review the construction of  
 the diagram  \eqref{k-diagram} from \cite{LamPyl}.
As noted in the introduction,
we will work with slightly modified versions of the objects discussed in \cite{LamPyl}, involving a formal parameter $\beta$. 
Setting $\beta=1$ recovers Lam and Pylyavskyy's original definitions, but one can
also go in the reverse direction by making appropriate variable substitutions.

\subsection{Small multipermutations}
\label{wqsym-sect}

We start with the object $\mMPR$ in the lower right corner of \eqref{k-diagram},
called the \defn{small multi-Malvenuto-Reutenauer Hopf algebra} in \cite{LamPyl}.

A \defn{word} is a finite sequence  of positive integers.
Let $v=v_1v_2\cdots v_m$ and $w=w_1w_2\cdots w_n$ be words.
When $S = \{s_1  < s_2 < \dots <s_m\} \subset [m+n]$  and $T
= \{t_1 < t_2<\dots <t_n\} = [m+n]\setminus S$,
define  $\shuffle_{S,T}(v,w):=u_1u_2\cdots u_{m+n}$
where $u_{s_i} := v_i$  and $u_{t_i} := w_i$.
The \defn{shuffle product} of $v$ and $w$ is 
$v \shuffle w := \sum \shuffle_{S,T}(v,w)$
where the sum is over all pairs $(S,T)$ of disjoint sets with $S\sqcup T =[m+n]$ where $|S| = m$ and $|T|=n$.
For example,  we have  $21 \shuffle 11 
= 3 \cdot 2111 + 2 \cdot 1211 + 1121.$

For $k \in \NN$, let
$w\uparrow k = (w_1+k)(w_2+k)\cdots (w_n+k)$.
If $w$ has $m$ distinct letters, then its standardization is the word
$\st(w) = \phi(w_1)\phi(w_2)\cdots \phi(w_n)$,
where $\phi$ is the unique order-preserving bijection $\{w_1,w_2,\dots,w_n\}\to[m]$.
A word $w$ is \defn{packed}
if $\st(w) = w$.

\begin{definition}
Let $\PWords$ denote the set of packed words and define 
$\mWQSym$ to be the linearly compact $\ZZ[\beta]$-module with $\PWords$ as a pseudobasis. 
\end{definition}

Define  
$\nabla : \mWQSym \htimes \mWQSym \to \mWQSym
$
and 
$
\Delta : \mWQSym \to \mWQSym \htimes \mWQSym
$
by 
\be\label{packed-prod-eq}
\nabla(v\otimes w) = v\shuffle (w \uparrow \max(v))
\quand
\Delta(w) = \sum_{i=0}^n \st(w_1\cdots w_i) \otimes \st(w_{i+1}\cdots w_n)
\ee
for $v \in \PWords$ and
$w=w_1w_2\cdots w_n \in \PWords$,  extending by linearity and continuity.
Write 
$\iota : \ZZ[\beta] \to \mWQSym$
and
$\epsilon : \mWQSym\to \ZZ[\beta]$
for the linear maps with
$\iota(1) = \emptyset$
and
$\epsilon(w) = \delta_{w\emptyset}$
for $w \in \PWords$.
These maps make
$\mWQSym$ into an LC-Hopf algebra 
 \cite[Prop.~3.11]{Marberg2018},
 called the Hopf algebra of \defn{word quasisymmetric functions}.

A \defn{small multipermutation} is a packed word with no equal adjacent letters.
Let $\Sm_n$ denote the set of such words $w$ with $\max(w) = n$
and define $\Sm_\infty := \bigsqcup_{n \in \NN} \Sm_n$.
Write $<_\m$ for the strongest partial order on $\PWords$ with $w_1\cdots w_i\cdots w_n <_\m w_1 \cdots w_{i}w_{i}\cdots w_n $.
Each lower set under $<_\m$ contains a unique minimal element, which belongs to $\Sm_\infty$.

\begin{definition}
Given  $v \in \Sm_\infty$, let 
$ [v]^{(\beta)}_\m := \sum_{ v \leq_\m w} \beta^{\ell(w)-\ell(v)} w \in \mWQSym$
where the sum is over packed words $w \in \PWords$,
and define $\mMPR$ to be the 
linearly compact $\ZZ[\beta]$-submodule of $\mWQSym$ with the elements $[v]^{(\beta)}_\m$ for $v \in \Sm_\infty$ as a pseudobasis.
\end{definition}


As $\mMPR$ is an LC sub-bialgebra of $\mWQSym$,
which is graded and connected, Takeuchi's formula \cite[Prop.~1.4.22]{GrinbergReiner} implies that its antipode preserves $\mMPR$. This observation lets us
recover the following statement, which is equivalent to \cite[Thms.~4.2 and 7.12]{LamPyl}. 

\begin{theorem}[{\cite{LamPyl}}]
\label{mmpr-thm}
The submodule $\mMPR$ is an LC-Hopf subalgebra of $\mWQSym$.
\end{theorem}

\cite[Thm.~4.2]{LamPyl}
constructs an LC-Hopf algebra over $\ZZ$ with $\Sm_\infty$ as a pseudobasis;
this object is isomorphic to the $\ZZ$-submodule of $\mMPR$ 
with $\left\{\beta^{\ell(v)} [v]^{(\beta)}_\m : v \in \Sm_\infty\right\}$ as a pseudobasis.

\subsection{Big multipermutations}

Next, we review the construction of the  \defn{big multi-Malvenuto-Reutenauer Hopf algebra}
from \cite{LamPyl},
which gives the dual object $\MMPR$ in the top right corner of \eqref{k-diagram}.

A \defn{set composition} 
is a sequence of pairwise disjoint nonempty sets $B=B_1B_2\cdots B_m$
with union $\bigsqcup_{i \in[m]} B_i = [n]$ for some $n \in \NN$;
in this case we define
 $\ell(B) := m$ and 
 $|B| := n$.

\begin{definition}
Let $\SetComp$ be the set of all set compositions
and define $\SetComp_n = \{ B \in \SetComp : |B| = n\}$ for $n \in \NN$.
Let $\MWQSym$ be the free $\ZZ[\beta]$-module with $\SetComp$ as a basis.
\end{definition}
There is a Hopf algebra structure on $\MWQSym$. 
Given $B=B_1B_2\cdots B_m \in \SetComp$ and $k \in \NN$,
let $k + B $ be the sequence of sets $(k+B_1)(k+B_2)\cdots (k+B_m)$.
For $S \subset \PP$,
define $B \cap S$ 
by removing any empty sets from $(B_1\cap S)(B_2\cap S)\cdots (B_m\cap S)$.
The product $\nabla : \MWQSym \otimes \MWQSym \to \MWQSym$
is the linear map with $\nabla(A\otimes B) = 
\sum
_{C \in A\bullet B} 
C$
where
\[A\bullet B := \left\{ C \in \SetComp_{m+n} : C \cap [m] = A\text{ and }C \cap (m+[n]) = m + B\right\}\]
for all $A \in \SetComp_m$ and $B \in \SetComp_n$.
For example,
the elements of $ \{1\}\{2\} \bullet \{1,2\}$ are
$\{1\}\{2\}\{3,4\}$,
$\{1\}\{2,3,4\}$,
$\{1\}\{3,4\}\{2\}$,
$\{1,3,4\}\{2\}$,
and
$\{3,4\}\{1\}\{2\}$.

If $B = B_1B_2\cdots B_m$ is a sequence of subsets of some totally ordered alphabet
and $n=|\bigcup_i B_i|$,
then we let $\st(B):=\phi(B_1) \phi(B_2)\cdots \phi(B_m)$
where $\phi$ is the   order-preserving bijection
$B_1 \cup B_2 \cup \dots \cup B_m \to [n]$.
The coproduct
$
\Delta : \MWQSym \to \MWQSym \otimes \MWQSym
$
is   the linear map with \[
\Delta(A) = \sum_{i=0}^m \st(A_1\cdots A_i) \otimes \st(A_{i+1}\cdots A_m)
\quad\text{for all $A=A_1A_2\cdots A_m \in \SetComp$.}\]
Write 
$\iota : \ZZ[\beta] \to \WQSym$
and
$\epsilon : \WQSym\to \ZZ[\beta]$
for the linear maps with
$\iota(1) =  \emptyset$
and
$\epsilon(A) = \delta_{A\emptyset}$ for $A \in \SetComp$.
These maps make  $\WQSym$ into a graded, connected Hopf algebra \cite[\S2.1]{NovelliThibon}.

Consider the following operations interchanging packed words and set compositions.
First, given $w=w_1w_2\cdots w_n \in \PWords$ with $\max(w)=m$,
define $w^*$ to be the set composition $A_1A_2\cdots A_m \in \SetComp_n$
with $A_i = \{ j \in [n] : w_j = i\}$.
Next, for $A =A_1A_2\cdots A_m \in \SetComp_n$, define $A^*$
to be the packed word $w_1w_2\cdots w_n$ with $w_j = i$ if $j \in A_i$.
Finally define $\langle\cdot,\cdot\rangle : \WQSym \times \mWQSym \to \ZZ[\beta]$ to be the unique bilinear
form,
continuous in the second coordinate, with
\be\label{form1-eq}
\langle A, w \rangle = \delta_{A,w^*}
\qquad\text{for all }A \in \SetComp\text{ and } w\in \PWords.
\ee
 This form is nondegenerate since  $w\mapsto w^*$ and $A \mapsto A^*$
are inverse bijections
$ \PWords \leftrightarrow \SetComp$.
One can also check directly that the relevant products and coproducts
 are compatible 
 in the sense of Section~\ref{completions-sect}.
Therefore $\WQSym$ and $\mWQSym$ are duals
with respect to \eqref{form1-eq}.

 A \defn{big multipermutation} is a set composition whose blocks never contain
 consecutive integers. 
Let $\SM_n$ be the set of big multipermutations $A$
with $|A| = n$,
and define $\SM_\infty :=\bigsqcup_{n\geq 0} \SM_n$.
The operations $w\mapsto w^*$ and $A \mapsto A^*$
restrict to inverse bijections
$ \Sm_n \leftrightarrow \SM_n$. 

Write $\equivM$ for the strongest partial order 
on set compositions with $A \equivM B$ whenever 
$B$ has a block containing $i$ and $i+1$
and $A = \st(B \cap \{1,\dots,i,i+2,\dots,n\})$.
Each lower set under $\equivM$ contains a unique minimal element, which is a big multipermutation.

\begin{definition}
Let $\cJM:=\ZZ[\beta]\spanning\left\{\beta^{|B|-|A|} A- B : A,B\in \SetComp\text{ with }A \equivM B\right\}$.
Denote the quotient module by $\MMPR:=\WQSym / \cJM$
and set
$[A]^{(\beta)}_\M := A + \cJM \in \MMPR$
for $A \in \SetComp$.
\end{definition}

The $\ZZ[\beta]$-module $\MMPR$ is free  with 
 basis  $\left\{[A]^{(\beta)}_\M: A \in \SM_\infty\right\}$.
One can check that $\cJM$ is the orthogonal complement of $\mMPR$,
which
implies the following results from \cite[\S7.2 and \S7.4]{LamPyl}:

\begin{theorem}[{\cite{LamPyl}}]
\label{form2-thm}
The submodule $\cJM$ is a Hopf ideal of $\WQSym$,
so $\MMPR$ is a quotient Hopf algebra.
This Hopf algebra is dual to $\mMPR$ via 
the bilinear form
$\langle\cdot,\cdot\rangle : \MMPR \times \mMPR\to \ZZ[\beta]$, 
continuous in the second coordinate, with
$
\langle [A]^{(\beta)}_\M, [w]^{(\beta)}_\m \rangle = \delta_{A,w^*}
$
for $A \in \SM_\infty$ and $w\in \Sm_\infty.$
\end{theorem}

The Hopf algebra $\MMPR$ is a very minor generalization of the Hopf algebra constructed in \cite[Thm.~7.5]{LamPyl},
which can be realized inside $\MMPR$ as the $\ZZ$-submodule 
spanned by $  \beta^{|A|} A$ for $A \in \SM_\infty$.

\subsection{Multifundamental quasisymmetric functions}

Here we review the construction   of an alternate pseudobasis for $\mQSym$,
which arises from viewing $\mMPR$ as a combinatorial LC-Hopf algebra.
The ideas in this section originate in \cite[\S5]{LamPyl}, but we follow the slightly different notational
conventions from \cite[\S3]{LM2019}.
 
For a composition $\alpha=(\alpha_1,\alpha_2,\dots,\alpha_k)\vDash N$
let 
$
I(\alpha) := \{\alpha_1, \alpha_1+\alpha_2,\dots, \alpha_1+\alpha_2 + \dots + \alpha_{k-1}\}
$.
Define $\PSet$ to be the set of nonempty, finite subsets of $\PP=\{1,2,3,\dots\}$. 
Given $S,T \in \PSet$, we write $S\preceq T$ if $\max(S) \leq \min(T)$ and $S \prec T$ if $\max(S) < \min(T)$.
\begin{definition}
\label{lbeta-def}
The \defn{multifundamental quasisymmetric function} of $\alpha\vDash N$ is  
$
 L^{(\beta)}_\alpha   := \sum_S
\beta^{|S| - N} x^S
\in \mQSym
$ where 
the sum is over all $N$-tuples $S=(S_1 \preceq S_2 \preceq \dots \preceq S_N)$ with $S_i \in \PSet$
and $S_i \prec S_{i+1}$ if $i \in I(\alpha)$, and where
$|S| := \sum_{i=1}^N |S_i|$ and $x^S := \prod_{i=1}^N \prod_{j \in S_i} x_{j}$.
\end{definition}

The  quasisymmetric functions $ L^{(\beta)}_\alpha$ are another pseudobasis for $\mQSym$ \cite[\S3.3]{LM2019}.

\begin{remark}\label{beta-rmk}
Setting $\beta=0$ transforms $L^{(\beta)}_\alpha$ to
the \defn{fundamental quasisymmetric functions} $L_\alpha := L^{(0)}_\alpha$ \cite[Def.~3.3.4]{LMW}.
Setting $\beta=1$  turns $L^{(\beta)}_\alpha$
into the quasisymmetric functions denoted $\tilde L_\alpha$ in \cite[\S5.3]{LamPyl}.
One recovers $L^{(\beta)}_\alpha$ from $\tilde L_\alpha$ via 
the identity $ L_\alpha^{(\beta)}= \beta^{-|\alpha|}\tilde L_\alpha(\beta x_1,\beta x_2,\dots)$,
which lets one rewrite many formulas in \cite{LamPyl} in terms of $ L_\alpha^{(\beta)}$.
For example, one can obtain explicit expressions
for the product $L^{(\beta)}_{\alpha'} L^{(\beta)}_{\alpha''}$ and coproduct $\Delta(L^{(\beta)}_{\alpha})$
in this way from  \cite[Props. 5.9 and 5.10]{LamPyl}.
\end{remark}

Write $\zeta_{<}$ for the continuous linear map $\mWQSym \to \ZZ[\beta]\llbracket t\rrbracket $
sending strictly increasing packed words $w$ to $t^{\ell(w)}$ and all other packed words to zero.
Then $\zeta_<$ is an algebra morphism with
$\zeta_{<}([w]^{(\beta)}_\m) = \zeta_{<}(w)$ for
all $w \in \Sm_\infty$.
The \defn{descent set} of a word $w=w_1w_2\cdots w_n$ is given by $\Des(w) := \{ i \in [n-1] : w_i > w_{i+1}\}$.
We write $\cC(w)$ 
for the composition 
of $\ell(w)$
with $I(\cC(w)) = \Des(w)$.
The as yet unmotivated definition of  $L^{(\beta)}_\alpha$ is algebraically natural in view of the following:

\begin{theorem}
\label{tl-thm}
The continuous linear map
with $[w]^{(\beta)}_\m \mapsto L^{(\beta)}_{\cC(w)}$ for all $w \in \Sm_\infty$
is
the unique morphism of combinatorial LC-Hopf algebras
$(\mMPR,\zeta_{<}) \to (\mQSym,\zetaq)$.
\end{theorem}

\begin{proof} The claim that this map is a morphism of LC-bialgebras (and therefore also of LC-Hopf algebras) is equivalent to \cite[Thm.~5.11]{LamPyl}.
Choose $w \in \Sm_\infty$ and set $\alpha := \cC(w)$ and $N := |\alpha| = \ell(w)$.
In view of Theorem~\ref{abs-thm}, we just need
to check that $\zeta_{<}([w]^{(\beta)}_\m) = \zetaq(L^{(\beta)}_{\alpha})$.
As $\zetaq$ sends $x_1\mapsto t$ and $x_i\mapsto 0$ for $i>0$,
we either have $\zetaq(L^{(\beta)}_{\alpha}) = t^N$ if the $N$-tuple $S = (\{1\} \preceq \{1\} \preceq \dots \preceq \{1\})$
satisfies the conditions in Definition~\ref{lbeta-def}, or else  $\zetaq(L^{(\beta)}_{\alpha}) =0$.
This means that $\zetaq(L^{(\beta)}_{\alpha}) = t^N = \zeta_{<}([w]^{(\beta)}_\m)$  
if $I(\alpha) = \Des(w)$ is empty and otherwise
$\zetaq(L^{(\beta)}_{\alpha}) = 0 = \zeta_{<}([w]^{(\beta)}_\m)$ as needed.
\end{proof}

\subsection{Noncommutative symmetric functions}

We now review the construction from \cite{LamPyl} of the  \defn{multi-noncommutative symmetric functions} $\MNSym$
 in the top row of \eqref{k-diagram}.
The \defn{descent set} of a set composition $A = A_1 A_2\cdots A_m \in \SetComp_n$
is given by
$\Des(A) := \{ i  \in [n-1]: i+1 \in A_j\text{ and }i \in A_k\text{ for any indices }j<k\}.$
One has $\Des(A) = \Des(A^*)$.

\begin{definition}
For a composition $\alpha \vDash n$,
define
$
\tR_\alpha := \sum_{\Des(A) =I(\alpha) } [A]^{(\beta)}_\M \in \MMPR
$
where the sum is over big multipermutations $A \in \SM_n$.
These sums are linearly independent, and we
 define $\MNSym$ to be the free $\ZZ[\beta]$-submodule of $\MMPR$
with $\left\{\tR_\alpha : \alpha \text{ is a composition} \right\}$ as a basis.
\end{definition}

Recall that we have a form 
$\langle\cdot,\cdot\rangle : \MMPR \times \mMPR\to \ZZ[\beta]$
from Theorem~\ref{form2-thm}.
Evidently if $\alpha$ is a composition and $w \in \Sm_\infty$ has $\gamma = \cC(w)$  
then  $\langle \tR_{\alpha},[w]^{(\beta)}_\m\rangle= \delta_{\alpha\gamma}$.
We reuse the symbol $\langle \cdot,\cdot\rangle $ 
to denote the bilinear form $\langle \cdot,\cdot\rangle : \MNSym \times \mQSym \to \ZZ[\beta]$, continuous in the second coordinate,
 with
$ \langle \tR_{\alpha}, L^{(\beta)}_{\gamma} \rangle = \delta_{\alpha\gamma}$ for all $\alpha$ and $\gamma$.
The following is equivalent to \cite[Thm.~8.4]{LamPyl}:

\begin{theorem}[{\cite{LamPyl}}]
\label{form3-thm}
The module $\MNSym$ is a Hopf subalgebra of $\MMPR$. This subalgebra is the Hopf algebra 
dual to $\mQSym$ via $\langle \cdot,\cdot\rangle $, and the  map $\mMPR \to \mQSym$ 
with 
 $[w]^{(\beta)}_\m \mapsto L^{(\beta)}_{\cC(w)}$ for all $w \in \Sm_\infty$
from Theorem~\ref{tl-thm}
is the morphism adjoint to 
the inclusion
$\MNSym \hookrightarrow \MMPR$.
\end{theorem}

%
%
The elements $\tR_{n}:=\tR_{(n)}$ for $n \in \PP$
 freely generate $\MNSym$ as an algebra \cite[Prop.~8.3]{LamPyl},
and one can view $\MNSym$ as a graded connected Hopf algebra in which $\tR_{n}$ has degree $n$.
 In fact, $\MNSym$ is
isomorphic to the usual  \defn{noncommutative symmetric functions} $\NSym$, just defined with scalar ring $\ZZ[\beta]$, by
 \cite[Prop.~8.5]{LamPyl}.

\subsection{Symmetric functions}\label{sym-sect}

 A \defn{symmetric function} in $\ZZ[\beta]\llbracket x_1,x_2,\dots\rrbracket $ is a power series 
 that is invariant under permutations of the $x_i$ variables.
The first column of 
 \eqref{k-diagram} contains these familiar  power series:
  
 \begin{definition}
Define $\MSym$ to be the Hopf subalgebra of symmetric functions of bounded degree in $\QSym$.
Let $\mSym$ be the LC-Hopf subalgebra of all symmetric functions in $\mQSym$.
\end{definition}

Let 
$\{s_\lambda\}$ denote the basis of
Schur functions for $\MSym$, indexed by partitions $\lambda$.
It is well-known that $\MSym$ and $\mSym$ are dual Hopf algebras with respect to the nondegenerate bilinear form $\langle\cdot,\cdot\rangle : \MSym \times \mSym \to \ZZ[\beta]$,
 continuous in the second coordinate, that has
$
\langle s_\lambda,s_\mu \rangle = \delta_{\lambda\mu}$.

Lam and Pylyavskyy \cite[Thm.~9.15]{LamPyl} show that $\mSym$ and $\MSym$ have another pair of dual bases given respectively by
 the \defn{stable Grothendieck polynomials} $\{G^{(\beta)}_\lambda\}$
and the \defn{dual stable Grothendieck polynomials} $\{g^{(\beta)}_\lambda\}$,
which satisfy $\langle g^{(\beta)}_\lambda, G^{(\beta)}_\mu\rangle =\delta_{\lambda\mu}$.
 These ``polynomials'' are symmetric generating functions for \defn{semistandard set-valued tableaux}
 and \defn{reverse plane partitions} of shape $\lambda$, respectively.
 
 Since in this article we will never need to work with $G^{(\beta)}_\lambda$ and $g^{(\beta)}_\lambda$ directly,
 we omit their definitions.
 If one does require precise definitions that follow our notational conventions, one 
should adopt the formulas in \cite[Thm.~4.6]{Yeliussizov2019} with $\beta$ replaced by $-\beta$.

\begin{theorem}
\label{sym-dual-thm}
The map $\MNSym \to \MSym$ adjoint to the inclusion
$\mSym \hookrightarrow \mQSym$
relative to the forms 
$\langle\cdot,\cdot\rangle$
 is the algebra morphism
with $\tR_{n} \mapsto g^{(\beta)}_{n} := g^{(\beta)}_{(n)}$ for all $n \in \PP$.
\end{theorem}

\begin{proof}
The elements $\tR_{n}$
 freely generate $\MNSym$ \cite[Prop.~8.3]{LamPyl},
 so there is a unique algebra morphism $\MNSym \to \MSym$
 with $\tR_n \mapsto g^{(\beta)}_{n} $ for $n \in \PP$. 
To show that this is the adjoint to $\mSym \hookrightarrow \mQSym$,
it suffices to check that $\langle \tR_{n}, G^{(\beta)}_\lambda \rangle= \delta_{(n),\lambda}$ for all partitions $\lambda$,
as we already know this is the value of $\langle g^{(\beta)}_{n}, G^{(\beta)}_\lambda \rangle $.
As we have $ \langle \tR_{\alpha}, L^{(\beta)}_{\gamma} \rangle = \delta_{\alpha\gamma}$, the desired identity can be deduced from \cite[Eq.\ (3.10)]{LM2019},
which gives the expansion of $G^{(\beta)}_\lambda$ into $L^{(\beta)}_\gamma$'s.
\end{proof}

 \cite[Thm.~9.13]{LamPyl} computes the 
 image of $\tR_\alpha$ 
 under the adjoint map $\MNSym \to \MSym$;
 this  turns out to be a dual stable Grothendieck polynomial indexed by a specific skew ribbon shape.

\section{Shifted $K$-theoretic Hopf algebras}\label{main-sect}

We now turn to the shifted analogues of the diagram \eqref{k-diagram} provided in \eqref{shk-diagram}.
We start by 
describing two shifted analogues of $\mQSym$
in Section~\ref{multipeak-sect}.
In Section~\ref{MPeak-sect} we investigate the duals of these LC-Hopf algebras, which provide
$K$-theoretic analogues of the \defn{peak algebra} studied in \cite{Nyman,Schocker}.
Sections~\ref{shifted-sym-sect} and \ref{pq-sect2} 
give an overview of the four (LC-)Hopf algebras of symmetric functions
on the left sides of the two diagrams in.\eqref{shk-diagram}.
In Sections~\ref{antipode-sect} 
we derive several antipode formulas, and then in
Section~\ref{open-sect} we conclude with a survey of open problems.

\subsection{Multipeak quasisymmetric functions}\label{multipeak-sect}

Our first task is to define the shifted analogues of $\mQSym$, which are displayed as
$\mQSymQ$ and $\mQSymP$ in \eqref{shk-diagram}.
This material is partly review from \cite{LM2019}.

For $i \in \ZZ$ let $i' := i - \frac{1}{2}$ so that $\frac{1}{2} \ZZ = \{ \dots < 0' < 0 < 1' < 1 < \dots\}$.
Define $\MSet$ to be the set of finite, nonempty subsets of $\MM = \{1'<1<2'<2< \dots\}.$
We again write $S\prec T$ if $\max(S) < \min(T)$ and $S \preceq T$ if $\max(S) \leq \min(T)$
for $S,T \in \MSet$. 
Let 
$x^S := \prod_{i=1}^N \prod_{j \in S_i} x_{\lceil j\rceil }$ and $ |S| := \sum_{i=1}^N |S_i|$
for any sequence $S=(S_1,S_2,\dots,S_N)$ with $S_i \in \MSet$.
A \defn{peak composition} is a composition $\alpha$ 
with $\alpha_i \geq 2$ for $1\leq i < \ell(\alpha)$.
Recall that $I(\alpha) = \{\alpha_1, \alpha_1+\alpha_2, \dots\} \setminus \{|\alpha|\}$. 

\begin{definition}\label{peak-def}
Suppose $\alpha \vDash N$ is a peak composition. Define
$
K^{(\beta)}_\alpha   := \sum_{
S
}
\beta^{|S|-N} x^S
$
where the sum is
 over all $N$-tuples $S=(S_1\preceq S_2 \preceq \dots\preceq S_N)$ of sets $S_i \in \MSet$
with
\be\label{peak-def-eq1}
 S_i \cap S_{i+1} \subset \{1',2',3',\dots\}\text{ if $i \in I(\alpha)$}
\quand
S_i \cap S_{i+1} \subset \{1,2,3,\dots\}\text{ if $i \notin I(\alpha)$.}
\ee
Define $\bK^{(\beta)}_\alpha   := \sum_{
S
}
\beta^{|S|-N} x^S$ where the sum is over the subset of such $N$-tuples $S$
also satisfying
\be\label{peak-def-eq2}
S_{i+1} \subset \{1,2,3,\dots\}\text{ if $i \in \{0\} \sqcup I(\alpha)$.}
\ee
Let $\mQSymQ $ and $\mQSymP$, respectively,
be the  LC-$\ZZ[\beta]$-modules
with  $\{ K^{(\beta)}_\alpha\}$ and $\{ \bK^{(\beta)}_\alpha\}$ (where
$\alpha$ ranges over all peak compositions)
as respective pseudobases.
\end{definition}

The power series $K^{(\beta)}_\alpha$ and $\bK^{(\beta)}_\alpha$ were introduced in \cite{LM2019}
in the context of an ``enriched'' theory of \defn{set-valued $P$-partitions}.
Setting $\beta=0$ transforms $K^{(\beta)}_\alpha$ and $\bK^{(\beta)}_\alpha$ to 
the \defn{peak quasisymmetric functions} defined in \cite[\S3]{Stembridge1997a},
and this implies that   $\{ K^{(\beta)}_\alpha\}$ and $\{ \bK^{(\beta)}_\alpha\}$ are linearly independent.
 However, $K^{(\beta)}_\alpha$ and $\bK^{(\beta)}_\alpha$ are typically not linear combinations 
 (even using rational coefficients and infinitely many terms) of
 the functions $K_\alpha := K^{(0)}_\alpha$
 and 
 $\bK_\alpha := \bK^{(0)}_\alpha$ from \cite[\S3]{Stembridge1997a}.

Both $\mQSymQ$ and $\mQSymP$ are LC-Hopf subalgebras of $\mQSym$,
and
if $\mQSymQ_{\QQ[\beta]}$ is the LC-Hopf algebra defined over $\QQ[\beta]$
 with $\{ K^{(\beta)}_\alpha\}$ as a pseudobasis,
then we have $\mQSymP = \mQSymQ_{\QQ[\beta]} \cap \mQSym\supsetneq \mQSymQ$ \cite[Thm.~4.19]{LM2019}.
More concretely, it holds that
\be\label{417-eq} K^{(\beta)}_\alpha  = \sum_{\delta \in \{0,1\}^\ell} 2^{\ell-|\delta|} \beta^{|\delta|} \oK^{(\beta)}_{\alpha+\delta}
\quand \oK^{(\beta)}_\alpha  =\sum_{\delta \in (\NN)^\ell} 2^{-\ell-|\delta|}(-\beta)^{|\delta|} K^{(\beta)}_{\alpha+\delta}\ee
for any peak composition $\alpha $ with $\ell=\ell(\alpha)$ parts, where $|\delta| := \sum_{i=1}^\ell \delta_i$  \cite[Cor.~4.17]{LM2019}.
When $\beta=0$, the Hopf algebras $\mQSymQ$ and $\mQSymP$ reduce to (the completions of) the ones denoted
$\mathbf{\Pi}$ and $ \bar{ \mathbf\Pi}$ in 
 \cite{BMSW,Stembridge1997a},
 which have been further studied in a number of places (see, e.g., \cite{BilleraHsiaoWill,Hsiao,HsiaoPetersen,Li2016,Petersen}).

Recall the definition of $\zeta_{<}:\mWQSym \to \ZZ[\beta]\llbracket t\rrbracket $ 
and write 
$\zeta_{>}:\mWQSym \to \ZZ[\beta]\llbracket t\rrbracket $
for the continuous linear map 
whose value at a packed word $w=w_1w_2\cdots w_n \in \PWords$ is
\be
\zeta_>(w_1w_2\cdots w_n) := \zeta_<(w_n \cdots w_2w_1)
= 
\begin{cases} t^n &\text{if }w_1>w_2>\dots> w_m \\ 0&\text{otherwise}.\end{cases}
\ee
This is an algebra morphism
with $\zeta_{>}([w]^{(\beta)}_\m) = \zeta_{>}(w)$ for $w \in \Sm_\infty$.
By Theorem~\ref{abs-thm} there is a unique morphism of combinatorial LC-Hopf algebras
$(\mMPR,\zeta_{>}) \to (\mQSym,\zetaq)$.
Although this map is different from the one in Theorem~\ref{tl-thm},
it also sends $\{ [w]^{(\beta)}_\m : w \in \Sm_\infty\}$ 
to the pseudobasis of multifundamental quasisymmetric functions $\{ L^{(\beta)}_\alpha\}$; see the proof of \cite[Prop.~6.3]{LM2019}.

To construct something new, we consider the \defn{convolution} of the maps $\zeta_{>}$ and $\zeta_{<}$ defined by
the formula
$ \zeta_{>|<} := \nabla_{\ZZ[\beta]} \circ (\zeta_> \htimes \zeta_<)\circ \Delta : \mWQSym \to \ZZ\llbracket t\rrbracket . $
This is a continuous algebra morphism $ \mWQSym \to \ZZ\llbracket t\rrbracket $.
If $w=w_1w_2\cdots w_n \in \PWords$ then 
\be \zeta_{>|<}(w)=
\begin{cases}
2 t^n &\text{if $w_1>\dots > w_i <\dots < w_n$ for some $i \in [n]$} \\
t^n &\text{if $w_1>\dots > w_i = w_{i+1} < \dots < w_n$ for some $i \in [n-1]$} \\
1 & \text{if }n=0\\
0&\text{otherwise}.\end{cases}
\ee
It follows that if $w=w_1w_2\cdots w_n \in \Sm_\infty$ then
\be \zeta_{>|<}([w]^{(\beta)}_\m) = \begin{cases}
t^n(2+\beta t) &\text{if $w_1>\dots > w_i < \dots < w_n$ for some $i \in [n]$} \\
1 &\text{if } n=0\\
0&\text{otherwise}.
\end{cases} 
\ee

For any composition $\alpha\vDash n$, let $\Lambda(\alpha)$ be the 
unique peak composition of $n$ satisfying $ I(\Lambda(\alpha)) = \{ i \in I(\alpha) : 0 < i-1 \notin I(\alpha)\}.$
For example, one has $\Lambda((1,2,1,1,1,3,1)) = (3,6,1)$. 
Then define
$\ttheta : \mQSym \to \mQSymQ$
to
be the continuous linear map with 
\be\label{ttheta-eq} 
\ttheta(L^{(\beta)}_\alpha) = K^{(\beta)}_{\Lambda(\alpha)}
\quad\text{for all compositions $\alpha$.}
\ee
By \cite[Cor.~4.22]{LM2019},
this map is a surjective morphism of LC-Hopf algebras.

The \defn{peak set} of a word $w=w_1w_2\cdots w_n$ is 
$\PeakSet(w) := \{1<i<n : w_{i-1} < w_i > w_{i+1}\}.$ 
Let $\cD(w)$ be the unique peak composition $\alpha\vDash\ell(w)$ with $I(\alpha) = \PeakSet(w)$.
If $w \in \Sm_\infty$ then 
\be\label{peakset-eq}
\PeakSet(w) = \{ i \in \Des(w) : 0 < i-1 \notin \Des(w)\}\quad\text{so}\quad 
\Lambda(\cC(w))=   \cD(w),\ee
and we have
$ \zeta_{>|<}([w]^{(\beta)}_\m) \neq 0$ if and only if $\PeakSet(w) = \varnothing$, in which case $\ell(\cD(w))\leq 1$.
The multipeak quasisymmetric functions are motivated algebraically by this analogue of Theorem~\ref{tl-thm}:

\begin{theorem}\label{peak-tl-thm}
The continuous linear map with $[w]^{(\beta)}_\m \mapsto K^{(\beta)}_{\cD(w)}$ for $w \in \Sm_\infty$
is the unique morphism of 
combinatorial LC-Hopf algebras 
$(\mMPR,\zeta_{>|<}) \to (\mQSym, \zetaq)$.
\end{theorem}

\begin{proof}
If $w \in \Sm_\infty$ then $\PeakSet(w) =  \{ i \in \Des(w) : 0 < i-1 \notin \Des(w)\}$ and  $\cD(w) = \Lambda(\cC(w))$ since
 $I( \Lambda(\cC(w))) = \{ i \in I(\cC(w)) : 0 < i-1 \notin I(\cC(w))\} =  \{ i \in \Des(w) : 0 < i-1 \notin \Des(w)\} .$
Our map $\Psi : \mMPR \to \mQSym$ is thus the composition of 
$\Phi : (\mMPR,\zeta_{<}) \to (\mQSym, \zetaq)$ from Theorem~\ref{tl-thm} and $\ttheta : \mQSym \to \mQSymQ$,
so $\Psi$ is at least a morphism of LC-Hopf algebras. 

It remains to check that $\zeta_{>|<} = \zetaq\circ \Psi$. For this,  it 
suffices to show that if $\alpha\vDash N$ is a peak composition then
  $ \zetaq(K^{(\beta)}_\alpha) = t^{|\alpha|}(2+\beta t)$ if $I(\alpha) = \varnothing$ and $ \zetaq(K^{(\beta)}_\alpha) =0$ otherwise.
Recall that $\zetaq$ corresponds to setting $x_1=t$ and $x_i=0$ for $i>1$.
Thus
$
\zetaq(K^{(\beta)}_\alpha)  = \sum_{S} \beta^{|S|-N} t^{|S|}
$ where the sum is over all weakly increasing $N$-tuples of sets $S= (S_1\preceq S_2 \preceq \cdots \preceq S_N)$ with $\varnothing\neq S_i \subseteq \{1'<1\}$,  $S_i \cap S_{i+1} \subseteq \{1'\}$ for $i \in I(\alpha)$, and $S_i \cap S_{i+1} \subseteq \{1\}$ for $i \notin I(\alpha)$. 

If $I(\alpha) \neq \varnothing$,  then $\alpha_1 \geq 2$ since $\alpha$ is a peak composition, so there are no such tuples $S$ since we must have $S_{i-1} \cap S_i \subset \{1\}$ and $S_i \cap S_{i+1} \subset \{1'\}$ for $i \in I(\alpha)$. Thus $\zetaq(K^{(\beta)}_\alpha) = 0$ if $I(\alpha) \neq \varnothing $ as claimed.
On the other hand, if $I(\alpha) = \varnothing$, then we have $S_i \cap S_{i+1} \subseteq \{1\}$ for all $i \in [N-1]$, so there
are only three possibilities for $S$, given by 
$( \{1\}, \{1\},\dots,\{1\})$, $( \{1'\}, \{1\},\dots,\{1\})$, and  $( \{1',1\}, \{1\},\dots,\{1\})$,
so we have $ \zetaq(K^{(\beta)}_\alpha) = t^{|\alpha|}(2+\beta t)$ as needed.
\end{proof}

At this point it is useful to describe the product and coproduct in $\mMPR$ more explicitly.
Recall the definition of $\leq_\m$ from Section~\ref{wqsym-sect}.
Given a word $w$ and a set $S$, define $w\cap S$ to be the subword of $w$ formed by omitting all letters not in $ S$.
Then, as explained in \cite[\S4]{LamPyl}, one has
\be\label{mMPR-prod-eq}
 \nabla([w']^{(\beta)}_\m\otimes [w'']^{(\beta)}_\m) = \sum_w  \beta^{\ell(w) - \ell(w')-\ell(w'')}[w]^{(\beta)}_\m
\qquad\text{for $w' \in \Sm_m$ and $w'' \in \Sm_n$}\ee
where the sum is over all $w \in \Sm_{m+n}$ such that $w' \leq_\m w \cap [m] $ and $w'' \uparrow m \leq_\m  w \cap (m+[n])$.
Similarly  if
we
fix $w =w_1w_2\cdots w_n \in \Sm_\infty$ and define $\llbracket w\rrbracket ^{(\beta)}_\m := [\st(w)]^{(\beta)}_\m$,
then
\be\label{mMPR-coprod-eq}
\Delta([w]^{(\beta)}_\m) = \sum_{i=0}^n  \llbracket w_1\cdots w_i\rrbracket ^{(\beta)}_\m  \otimes \llbracket w_{i+1}\cdots w_n\rrbracket ^{(\beta)}_\m
+ \beta \sum_{i=1}^n \llbracket w_1\cdots w_i\rrbracket ^{(\beta)}_\m  \otimes \llbracket w_i\cdots w_n\rrbracket ^{(\beta)}_\m.
\ee
These formulas follow directly from the definitions of $\mWQSym$ and $[w]^{(\beta)}_\m$.
Using Theorem~\ref{form2-thm}, one can translate these identities by duality to product and coproduct formulas
for $\MMPR$; see \cite[\S7.1]{LamPyl}.
On the other hand, invoking Theorem~\ref{peak-tl-thm} leads to the following formulas for $\mQSymQ$:

\begin{proposition}\label{K-prod-prop}
Suppose $\alpha' $ and $\alpha'' $ are peak compositions. Choose any
 $w' ,w'' \in \Sm_\infty$ with $\cD(w') = \alpha'$ and $\cD(w'') =\alpha''$,
 and set $m=\max(\{0\} \cup w')$ and $n=\max(\{0\} \cup w'')$.
  Then
\[
K^{(\beta)}_{\alpha'} K^{(\beta)}_{\alpha''} = \sum_w \beta^{\ell(w) - |\alpha'| - |\alpha''|} K^{(\beta)}_{\cD(w)}
\]
where the sum is over all $w \in \Sm_{m+n}$ such that $w' \leq_\m w \cap [m]$ and $w'' \uparrow m \leq_\m w \cap (m + [n])$.
Additionally, if $\alpha$ is a peak composition and $w=w_1w_2\cdots w_\ell \in\Sm_\infty$ has  
$\cD(w) = \alpha$, then
\[
\Delta(K^{(\beta)}_{\alpha}) = \sum_{i=0}^\ell K^{(\beta)}_{\cD(w_1\cdots w_i)} \otimes K^{(\beta)}_{\cD(w_{i+1}\cdots w_\ell)}
+
\beta \sum_{i=1}^\ell K^{(\beta)}_{\cD(w_1\cdots w_i)} \otimes K^{(\beta)}_{\cD(w_{i}\cdots w_\ell)}.
\]
\end{proposition}

\begin{proof}
Apply the morphism in Theorem~\ref{peak-tl-thm} to both sides of  \eqref{mMPR-prod-eq} and \eqref{mMPR-coprod-eq}.
\end{proof}


In principle one can also
 compute products and coproducts in the pseudobasis $\{ \bK^{(\beta)}_\alpha\}$
by combining the formulas in Proposition~\ref{K-prod-prop} with the change-of-basis identities in \eqref{417-eq}.

\begin{example}\label{K-prod-ex}
If $\alpha' = \alpha'' = (1)$ then $K^{(\beta)}_{\alpha'} K^{(\beta)}_{\alpha''}
 $
is the sum $ \sum_{w} \beta^{\ell(w) - 2} K^{(\beta)}_{\cD(w)}$ over all words  $w\in \{ 12, 21, 121, 212, 1212, 2121, 12121,21212, \dots\},$ 
so there is an infinite product expansion 
\[K^{(\beta)}_{(1)} K^{(\beta)}_{(1)} 
=
 2 K^{(\beta)}_{(2)} + \beta K^{(\beta)}_{(2,1)} + \beta K^{(\beta)}_{(3)}+ \beta ^2 K^{(\beta)}_{(2,2)}+ \beta^2 K^{(\beta)}_{(3,1)}+ 
 \beta^3 K^{(\beta)}_{(2,2,1)}+\beta^3 K^{(\beta)}_{(3,2)} + \dots .
\] However, there is a finite coproduct expansion
\[ \Delta(K^{(\beta)}_{(2)}) = \Delta(K^{(\beta)}_{\cD(12)}) =
1 \otimes K^{(\beta)}_{(2)} + K^{(\beta)}_{(1)}\otimes K^{(\beta)}_{(1)}
+ 1\otimes K^{(\beta)}_{(2)}
+
\beta\(  K^{(\beta)}_{(1)}\otimes K^{(\beta)}_{(2)}
+
  K^{(\beta)}_{(2)}\otimes K^{(\beta)}_{(1)}\).
 \]
\end{example}

\subsection{Multipeak noncommutative symmetric functions}\label{MPeak-sect}

We now consider the duals of $\mQSymQ$ and $\mQSymP$.
Fix a set composition $A = A_1 A_2\cdots A_m \in \SetComp_n$,
so that the union of the blocks of $A$ is $[n]$.
Recall that $i \in [n-1]$ belongs to $\Des(A)$
if and only if the block of $A$ containing $i$ is after the block containing $i+1$. 

The \defn{peak set} of $A$ is
$\PeakSet(A) := \{1<i<n: i -1 \notin \Des(A)\text{ and } i \in \Des(A) \}.$
If $A $ belongs to $ \SM_n$ (so that none of its blocks contain consecutive integers)
 then one has $i \in \PeakSet(A)$ precisely when $1< i <n$ and the block of $A$ containing $i$ is after the 
blocks containing $i-1$ and $i+1$.
Even when $A \notin \SM_n$, the set $\PeakSet(A)$ is always equal to $ I(\alpha)$ for some peak composition $\alpha \vDash n$.

\begin{definition}
For a peak composition $\alpha \vDash n$,
let
\[\tpeakR_\alpha := \sum_{\substack{A \in \SM_n \\  \PeakSet(A) =I(\alpha)}} [A]^{(\beta)}_\M  \in \MMPR.
\]
Then define $\MPeakP$ to be the free $\ZZ[\beta]$-module 
with $\left\{\tpeakR_\alpha : \alpha\text{ is a peak composition}\right\}$ as a basis.
\end{definition}

It also holds that $\tpeakR_\alpha =  \sum_{\Lambda(\gamma) =\alpha} \tR_\gamma$
where the sum is over $\gamma \vDash |\alpha|$, so $\MPeakP\subseteq\MNSym$.
Define
$ [\cdot,\cdot ] : \MPeakP \times \mQSymQ \to \ZZ[\beta] $
to be the nondegenerate bilinear form, continuous in the second coordinate, that has 
$[ \tpeakR_\alpha, K^{(\beta)}_\gamma] = \delta_{\alpha\gamma}$
for all peak compositions $\alpha$ and $\gamma$.
Below, let $\langle \cdot,\cdot\rangle : \MNSym \times \mQSym \to \ZZ[\beta]$
be as in Theorem~\ref{form3-thm} 
and recall the definition of $\ttheta $ from \eqref{ttheta-eq}.

\begin{lemma}\label{[]-lem}
If $f \in \MPeakP$ and $g \in \mQSym$
then $\left[ f, \ttheta(g)\right] = \langle f,g\rangle$. 
\end{lemma}

\begin{proof}
We may assume that $f = \tpeakR_\alpha$ and $g = L^{(\beta)}_\gamma$
 for a peak composition $\alpha$ and a composition $\gamma$. 
Then the desired identity is clear by comparing the definitions of $\ttheta$ and $\tpeakR_\alpha$.
\end{proof}

\begin{theorem}\label{mpeakp-thm}
The module $\MPeakP$ is a Hopf subalgebra of $\MNSym$
and is the Hopf algebra dual to $\mQSymQ$ via   $[\cdot,\cdot]$.
The continuous linear map 
$\mMPR \to \mQSymQ$ 
from Theorem~\ref{peak-tl-thm}
sending
$[w]^{(\beta)}_\m \mapsto K^{(\beta)}_{\cD(w)}$  for all $w \in \Sm_\infty$
is the morphism adjoint to 
the inclusion
$\MPeakP \hookrightarrow \MMPR$.
\end{theorem}

\begin{proof}
Relative to the form in Theorem~\ref{form2-thm},
the set $\MPeakP$ is the orthogonal complement
of the kernel of the LC-Hopf algebra morphism 
$\mMPR \to \mQSymQ$ described in Theorem~\ref{peak-tl-thm}.
Therefore $\MPeakP$ is a Hopf subalgebra.
Lemma~\ref{[]-lem}, in view of Theorem~\ref{form3-thm}, implies
that the nondegenerate form $[\cdot,\cdot]$ respects the (co)product and (co)unit maps of 
$\MPeakP$ and $\mQSymQ$, so
$\MPeakP$  dual to $\mQSymQ$.
For the last assertion, 
we note that if $\alpha$ is a peak composition and $w \in \Sm_\infty$
then  
\[\langle\tpeakR_\alpha, [w]^{(\beta)}_\m\rangle
=\langle 
  \tpeakR_\alpha, L^{(\beta)}_{\cC(w)} \rangle 
    =[ 
  \tpeakR_\alpha, \ttheta(L^{(\beta)}_{\cC(w)} )]    
   =[ 
  \tpeakR_\alpha, K^{(\beta)}_{\cD(w)} ] 
  \]
  by Theorem~\ref{form3-thm} for the first equality,  
   Lemma~\ref{[]-lem} for the second, and \eqref{peakset-eq} for the third.
\end{proof}

We call $\MPeakP$ the \defn{multi-peak algebra}. 
This is a generalization of the \defn{peak algebra} carefully studied in \cite{Schocker} (see also \cite{AguiarBergeronNyman,AguiarNymanOrellana,BergeronHivertThibon,JingLi,Nyman}),
which coincides with $\MPeakP$ when $\beta=0$.

We can compute a product formula for the $\tpeakR_\alpha$-basis of $\MPeakP$.
Suppose
$\alpha$ and $\gamma$ are nonempty peak compositions of length $m$ and $n$.
Define 
$\alpha \vartriangleleft \gamma:= \alpha\gamma$ and 
\[
\ba
\alpha \vartriangleright \gamma &:=  (\alpha_1,\dots,\alpha_{m-1},\alpha_m+ \gamma_1,\gamma_2,\dots,\gamma_n),
\\
\alpha \circ \gamma &:=  (\alpha_1,\dots,\alpha_{m-1},\alpha_m+ \gamma_1-1,\gamma_2,\dots,\gamma_n).
\ea
\]
Additionally let
\be\ba
\alpha \blacktriangleright \gamma &:= \alpha \vartriangleright (1, \gamma_1-1,\gamma_2,\dots,\gamma_n) =  (\alpha_1,\dots,\alpha_{m-1},\alpha_m+1, \gamma_1-1,\gamma_2,\dots,\gamma_n),
\\
\alpha\bullet \gamma &:=  \alpha \circ (1, \gamma_1-1,\gamma_2,\dots,\gamma_n)  = (\alpha_1,\dots, \alpha_m, \gamma_1-1,\gamma_2,\dots,\gamma_n).
\ea\ee
If $n=1$ then $\alpha \blacktriangleright \gamma$ and $\alpha\bullet \gamma $ could be integer sequences ending in zero,
which we do not consider to be peak compositions.
We define $\tpeakR_\alpha := 0$ if $\alpha$ is not a peak composition.
The following result is a shifted analogue of the product formula \cite[Prop.~8.1]{LamPyl}
for the $\tR_{\alpha}$-basis of $\MNSym$.

\begin{proposition}\label{MPeakP-prod-prop}
Suppose $\alpha$ and $\gamma$ are nonempty peak compositions. Then
\[\tpeakR_{\alpha}  \tpeakR_{\gamma}
=  \tpeakR_{\alpha\blacktriangleright \gamma} + \tpeakR_{\alpha\vartriangleright \gamma} + \tpeakR_{\alpha \vartriangleleft \gamma} 
+ \beta\cdot \tpeakR_{\alpha \circ \gamma} +  \beta\cdot  \tpeakR_{\alpha \bullet \gamma}.
\]
\end{proposition}


  \begin{proof}
  Choose a word $w=w_1w_2\cdots w_n \in \Sm_\infty$.
By Proposition~\ref{K-prod-prop} and Theorem~\ref{mpeakp-thm}   we have 
\[
\ba
{[\tpeakR_{\alpha}  \tpeakR_{\gamma}, K^{(\beta)}_{\cD(w)}]} &= 
[\tpeakR_{\alpha} \otimes \tpeakR_{\gamma}, \Delta(K^{(\beta)}_{\cD(w)})] 
\\&=
\sum_{i=0}^n \delta_{\alpha,\cD(w_1\cdots w_i)} \delta_{\gamma,\cD(w_{i+1}\cdots w_n)}
+
\beta \sum_{i=1}^n \delta_{\alpha, \cD(w_1\cdots w_i)} \delta_{\gamma,\cD(w_{i}\cdots w_n)}.
\ea
\]
Now observe that one can have $\alpha=\cD(w_1\cdots w_i) $ and $\gamma=\cD(w_{i+1}\cdots w_n)$ 
precisely when either $i \in \PeakSet(w)$ and $\cD(w) = \alpha\vartriangleleft \gamma$,
$i+1 \in \PeakSet(w)$ and $\cD(w) = \alpha\blacktriangleright \gamma$,
or $\{i,i+1\} \cap \PeakSet(w) = \varnothing$ and $\cD(w) =\alpha\vartriangleright\gamma$.
Similarly, one can have 
$\alpha=\cD(w_1\cdots w_i) $ and $\gamma=\cD(w_{i}\cdots w_n)$
precisely when
 $i \in \PeakSet(w)$ and $\cD(w) = \alpha\bullet \gamma$
 or when
 $i \notin \PeakSet(w)$ and $\cD(w) = \alpha\circ \gamma$.
  \end{proof}

\begin{example}
Here are two examples of Proposition~\ref{MPeakP-prod-prop}.  All five terms appear in  
\[
\tpeakR_{(3,2,5,2)}  \tpeakR_{(4,2)} = \tpeakR_{(3,2,5,3,3,2)} + \tpeakR_{(3,2,5,6,2)}  + \tpeakR_{(3,2,5,2,4,2)} +\beta \cdot\tpeakR_{(3,2,5,5,2)} +\beta \cdot\tpeakR_{(3,2,5,2,3,2)}.
  \]
On the other hand, only three survive in 
\[
\ba \tpeakR_{(3,2,5,1)}  \tpeakR_{(4,2)} &= \tpeakR_{(3,2,5,2,3,2)} + \tpeakR_{(3,2,5,5,2)}  + \tpeakR_{(3,2,5,1,4,2)} +\beta \cdot\tpeakR_{(3,2,5,4,2)} +\beta\cdot \tpeakR_{(3,2,5,1,3,2)}
\\
&=
 \tpeakR_{(3,2,5,2,3,2)} + \tpeakR_{(3,2,5,5,2)}  +\beta\cdot \tpeakR_{(3,2,5,4,2)} .
\ea
  \]
 \end{example}

Below, we abbreviate by writing $\tpeakR_{n}:=\tpeakR_{(n)}$ for $n \in \PP$ and set $\tpeakR_0 := 1$.

\begin{lemma}\label{peak-free-lem}
If $n$ is a positive integer and $\delta_{[n\text{ is even}]} := | \{ n\} \cap \{2,4,6,\dots\}|$ then
\[ \tpeakR_1 \tpeakR_{n-1} \in \delta_{[n\text{ is even}]}  \cdot \tpeakR_n +   \sum_{i=2}^{n-1} (-1)^i \tpeakR_i \tpeakR_{n-i}  
+  \ZZ[\beta]\spanning\left\{ \tpeakR_\alpha : |\alpha| <  n\right\}.
\]
\end{lemma}

\begin{proof}
Let $\cI =   \ZZ[\beta]\spanning \{ \tpeakR_\alpha : |\alpha| <  n \}$.
Proposition~\ref{MPeakP-prod-prop}
implies 
$\tpeakR_1 \tpeakR_{n-1}  +\cI = \tpeakR_{n} +\tpeakR_{(2,n-2)} + \cI
$
for $n>1$, where $\tpeakR_{(n,0)} := 0$.
Likewise, if $2\leq i < n$ then
$
\tpeakR_{(i, n-i)} +\cI = \tpeakR_{i}\tpeakR_{n-i}- \tpeakR_{n}-\tpeakR_{(i+1,n-i-1)} + \cI
$.
We obtain the lemma by 
successively expanding the right hand side of the first identity using the second. 
\end{proof}

\begin{proposition}\label{MPeak-free-prop}
The set $\left\{\tpeakR_n: n =1,3,5,\dots\right\}$ freely generates $\MPeakP$ as a $\ZZ[\beta]$-algebra.
\end{proposition}

\begin{proof}
Suppose $\alpha=(\alpha_1,\alpha_2,\dots,\alpha_m)$ is a composition
and let $\bXi_\alpha := \tpeakR_{\alpha_1}  \tpeakR_{\alpha_2} \cdots  \tpeakR_{\alpha_m}$.
We say that $\alpha$ is \defn{odd} if $\alpha_i$ is an odd integer for each $i \in [m]$.
It suffices to show that the elements  $\bXi_\alpha$ form a $\ZZ[\beta]$-basis for $\MPeakP$ when $\alpha$ ranges over all odd compositions.

First let $<_{\mathrm{revlex}}$ be the partial order on compositions with $\gamma <_{\mathrm{revlex}} \alpha$ 
if $|\gamma| < |\alpha|$ or if $|\gamma| = |\alpha|$ and $\gamma$ exceeds $\alpha$ in lexicographic order.
It follows from Proposition~\ref{MPeakP-prod-prop}
that if $\alpha$ is any peak composition then $\bXi_\alpha \in \tpeakR_\alpha + \ZZ[\beta]\spanning\left\{ \tpeakR_\gamma : \gamma <_{\mathrm{revlex}} \alpha\right\}$.
Thus  the elements  $\bXi_\alpha$ form a basis for $\MPeakP$ at least when $\alpha$ ranges over all peak compositions.

If $\alpha$ is a peak composition then let $\mathsf{odd}(\alpha)$ be the odd composition formed by replacing 
each even part $\alpha_i$ by two consecutive parts $(1,\alpha_i-1)$. 
For example, $\mathsf{odd}((3,6,3,4,2)) = (3,1,5,3,1,3,1,1)$.
Let $<_{\mathrm{lex}}$ be the partial order on compositions with $\gamma <_{\mathrm{lex}} \alpha$ if $|\gamma| < |\alpha|$
or if $|\gamma| = |\alpha|$ and $\gamma$ precedes $\alpha$ lexicographically.
It follows by induction on $\ell(\alpha)$ using Lemma~\ref{peak-free-lem}
that
$\bXi_{\mathsf{odd}(\alpha)} \in \bXi_\alpha + \ZZ[\beta]\spanning\left\{ \bXi_\gamma : \gamma <_{\mathrm{lex}} \alpha\right\}$.
Since $\mathsf{odd}$ is a bijection from peak compositions to odd compositions,
we deduce that $\left\{\bXi_\alpha : \alpha\text{ is an odd composition}\right\}$
 is another $\ZZ[\beta]$-basis for $\MPeakP$ as desired.
\end{proof}

To describe the dual of $\mQSymP$
we need a variant of $\SM_n$.
 Define $\barSM_n$ to consist of the set compositions $A = A_1 A_2\cdots A_m \in \SetComp_n$
 such that if $\{i,i+1\} \subseteq A_j$ for $i \in [n-1]$ and $j \in [m]$ then 
 the union $A_1 \sqcup A_2 \sqcup \dots \sqcup A_{j}$ contains neither $i-1$ nor $i+2$.
Then $\SM_n \subseteq \barSM_n$, and 
 for $A \in \barSM_n$ it still holds that $i \in \PeakSet(A)$ if and only if 
 $1<i<n$ and $i$ appears in a block of $A$ after the blocks containing $i-1$ and $i+1$.
For example, the elements of $\barSM_4 - \SM_4$ with peak set $\{3\}$ are 
$
 \{1,2,4\}\{3\},$
$  \{1,2\}\{4\}\{3\},$ and
$   \{4\}\{1,2\}\{3\}.$
For any $A \in \SetComp_n$ 
let $\cons(A) := | \left\{ i \in [n-1] : 
\{i,i+1\}\text{ is a subset of some block of $A$}\right\}|$.
\begin{definition}
For a peak composition $\alpha\vDash n$
let
\[
\opeakR_\alpha  := \sum_{\substack{A \in \barSM_n \\ \PeakSet(A) = I(\alpha)}} 2^{\ell(\alpha) - \cons(A)} [A]^{(\beta)}_\M \in \MMPR 
  .\]
Then let $\MPeakQ$ be the free $\ZZ[\beta]$-module 
with $\left\{\opeakR_\alpha : \alpha\text{ is a peak composition}\right\}$ as a basis.
\end{definition}

For example, if $\alpha=(3,1)$ then we have 
\[
\ba \opeakR_{(3,1)} &= 4\cdot \tpeakR_{(3,1)}
+
2\cdot [\{1,2,4\}\{3\}]^{(\beta)}_\M +
2\cdot  [\{1,2\}\{4\}\{3\}]^{(\beta)}_\M +
  2\cdot [\{4\}\{1,2\}\{3\}]^{(\beta)}_\M
  \\&= 4\cdot \tpeakR_{(3,1)}
+
2\beta\cdot [\{1,3\}\{2\}]^{(\beta)}_\M +
2\beta\cdot  [\{1\}\{3\}\{2\}]^{(\beta)}_\M +
  2\beta\cdot [\{3\}\{1\}\{2\}]^{(\beta)}_\M 
  \\&= 4\cdot \tpeakR_{(3,1)} + 2\beta \cdot \tpeakR_{(2,1)}.
  \ea
   \]
The following shows that $\{\opeakR_\alpha\}$ is linearly independent so $\MPeakQ$ is well-defined.

\begin{lemma}\label{opeak-lem}
If $\alpha$ is a peak composition with $\ell=\ell(\alpha)$ then
\[
\opeakR_\alpha= \sum_{\delta \in \{0,1\}^\ell} 2^{\ell-|\delta|} \beta^{|\delta|} \tpeakR_{\alpha-\delta}
\quand
\tpeakR_\alpha= \sum_{\delta \in (\NN)^\ell} 2^{-\ell-|\delta|} (-\beta)^{|\delta|} \opeakR_{\alpha-\delta}
\]
where we set $\tpeakR_{\alpha-\delta} :=0$ and $\opeakR_{\alpha-\delta} :=0$ if $\alpha-\delta$  is not a peak composition.
\end{lemma}

\begin{proof}
The map that sends  $A \in  \SetComp$ to the unique $B \in \SM_\infty$ with $B \leq_\M A$
is a bijection 
\be\label{opeak-map}
\left\{ A \in \barSM_{|\alpha|} : \PeakSet(A) = I(\alpha)\right\} \to \bigsqcup_{\delta} \left\{ B \in \SM_{|\alpha|-|\delta|} : \PeakSet(B) = I(\alpha-\delta)\right\}\ee
where the union is over $\delta \in\{0,1\}^\ell$ such that $\alpha-\delta$ is a peak composition.
To construct the inverse map, define a \defn{valley} of a big multipermutation  $B \in \SM_\infty$
to be a number $a$ whose block in $B$ is not weakly after the blocks containing $a-1$ or $a+1$.
Each $B \in \SM_{|\alpha|-|\delta|}$ with $\PeakSet(B) = I(\alpha-\delta)$  
has exactly $\ell$ valleys $a_1<a_2<\dots<a_\ell$.
Form $A$ from such $B$ by replacing the valley $a_i$ by two numbers $a_i'<a_i$ whenever $\delta_i=1$ and then standardizing.
For example,
the valleys of $B=\{1,3\}\{2\}$ are $1$ and $3$ so if $\delta=(1,0)$ then this inverse map would give 
$ B=\{1,3\}\{2\} \mapsto \{1',1,3\}\{2\} \mapsto \{1,2,4\}\{3\} = A$.
The formula for $\opeakR_\alpha$ follows  since
if \eqref{opeak-map} sends $A \in \barSM_{|\alpha|}$ to    $B \in \SM_{|\alpha|-|\delta|}$ then $o(A) =|\delta|$ and $[A]^{(\beta)}_\M = \beta^{|\delta|} [B]^{(\beta)}_\M$.
Inverting this identity to get the formula for $\tpeakR_\alpha$ is straightforward.
\end{proof}

\begin{theorem}
There is a unique extension of $[\cdot,\cdot]$
to a bilinear form $\MPeakQ \times  \mQSymP \to \ZZ[\beta]$,
continuous in the second coordinate, with
 $[\opeakR_\alpha,\bK^{(\beta)}_\gamma] = \delta_{\alpha\gamma}$ 
for all $\alpha$ and $\gamma$.
Therefore $\MPeakQ$ is a Hopf subalgebra of $\MPeakP$
and is the Hopf algebra dual to $\mQSymP$ via  $[\cdot,\cdot]$.
\end{theorem}

\begin{proof}
The first claim follows by computing  $[\opeakR_\alpha,\bK^{(\beta)}_\gamma]  $ 
from 
the identity $[\tpeakR_\alpha,K^{(\beta)}_\gamma] := \delta_{\alpha\gamma}$ after
substituting the formulas in Lemma~\ref{opeak-lem} and \eqref{417-eq} for $\opeakR_\alpha$ and $\bK^{(\beta)}_\gamma$.
The second assertion holds as $\MPeakP$ and $\mQSymQ$ are already dual via $[\cdot,\cdot]$ by Theorem~\ref{mpeakp-thm}
and each element of $\mQSymP$ is a formal $\QQ[\beta]$-linear combination of elements of $\mQSymQ$  by \eqref{417-eq}.
\end{proof}

\begin{remark*}
It follows that the 
morphism  $\mMPR \to \mQSymP$ adjoint to 
the inclusion
$\MPeakQ \hookrightarrow \MMPR$ 
has the same formula 
$[w]^{(\beta)}_\m \mapsto K^{(\beta)}_{\cD(w)}$  for $w \in \Sm_\infty$
as the adjoint map in Theorem~\ref{mpeakp-thm}.
\end{remark*}

Define
$\alpha \ast \gamma :=
 (\alpha_1,\dots,\alpha_{m-1}, \alpha_m-2+\gamma_1,\gamma_2,\dots,\gamma_n)$
 for 
 nonempty peak compositions $\alpha$ and $\gamma$  of length $m$ and $n$.
Below, as usual, we set $\opeakR_\alpha = 0$ if $\alpha$ is not a peak composition.

\begin{proposition}\label{opeak-prod-prop}
Suppose $\alpha$ and $\gamma$ are nonempty peak compositions of length $m$ and $n$. Then
\[\opeakR_{\alpha}  \opeakR_{\gamma}
= 
  \opeakR_{\alpha\blacktriangleright \gamma} + 2\cdot \opeakR_{\alpha\vartriangleright \gamma} + \opeakR_{\alpha \vartriangleleft \gamma} 
+ (1 + r + s) \beta  \cdot \opeakR_{\alpha \circ \gamma} +  \beta \cdot \opeakR_{\alpha \bullet \gamma} + rs   \beta^2 \cdot \opeakR_{\alpha \ast \gamma}
\]
where  $r=1$ (respectively, $s=1$) if the sequence
$(\alpha_1,\dots,\alpha_{m-1},\alpha_m-1)$ (respectively, the sequence $(\gamma_1-1,\gamma_2, \dots,\gamma_n)$)
is a peak composition\footnote{This means  $r=1$ if the last part of $\alpha$ is greater than one,
and $s=1$ if $\gamma_1>2$ or $\gamma=(2)$.}
and otherwise $r=0$ (respectively, $s=0$).
\end{proposition}

\begin{proof}
For each $x,y\in\NN$ define
\[
(\alpha\mid x] := \sum_{\substack{\delta \in \{0,1\}^{m} \\ \delta_m=0}} 2^{m-1-|\delta|} \beta^{|\delta|} \tpeakR_{(\alpha_1,\dots,\alpha_{m-1}, x)-\delta}
\text{\ \ and\ \ }
[y\mid \gamma) := \sum_{\substack{\delta \in \{0,1\}^{n} \\ \delta_1=0 }} 2^{n-1-|\delta|} \beta^{|\delta|} \tpeakR_{(y, \gamma_2, \dots,\gamma_n)-\delta}
\]
so that $\opeakR_\alpha= 2\cdot (\alpha\mid \alpha_m]+ \beta\cdot (\alpha\mid \alpha_m-1]$
and
$\opeakR_\gamma= 2\cdot [\gamma_1\mid \gamma)+ \beta\cdot [\gamma_1-1\mid \gamma)$.
Also let
\[
\ba
(\alpha\mid x \mid \gamma) &:= \sum_{\substack{\delta \in \{0,1\}^{m+n-1} \\ \delta_m=0}} 2^{m+n-2-|\delta|} \beta^{|\delta|} \tpeakR_{(\alpha_1,\dots,\alpha_{m-1}, x,\gamma_2,\dots,\gamma_n)-\delta},
\\
(\alpha\mid x \mid y  \mid \gamma) &:= \sum_{\substack{\delta \in \{0,1\}^{m+n} \\ \delta_m=\delta_{m+1}=0 }} 2^{m+n-2-|\delta|} \beta^{|\delta|} \tpeakR_{(\alpha_1,\dots,\alpha_{m-1}, x, y, \gamma_2, \dots,\gamma_n)-\delta}.
\ea
\]
Our conventions mean that 
these summations are zero if $x=0$ or $y=0$.
By Proposition~\ref{MPeakP-prod-prop}  
\[
\ba 
(\alpha\mid x]\cdot [y\mid \gamma) = (\alpha\mid x+1 \mid y-1 \mid \gamma) &+  (\alpha\mid x + y \mid \gamma)  +  (\alpha\mid x \mid y \mid \gamma) 
\\&+ \beta\cdot  (\alpha\mid x + y-1 \mid \gamma)  + \beta \cdot (\alpha\mid x \mid y-1 \mid \gamma) 
\ea
 \]
 for any positive integers $x$ and $y$.
 The desired formula follows by using this identity to expand the right side of
 $
 \opeakR_\alpha  \opeakR_\gamma= \(2\cdot (\alpha\mid \alpha_m]+ \beta\cdot (\alpha\mid \alpha_m-1]\)  \(2\cdot [\gamma_1\mid \gamma)+ \beta \cdot [\gamma_1-1\mid \gamma)\)
$
and then combining terms. There are a large number of terms and a few different cases to consider (according to whether $\alpha_m=1$ or $\gamma_1= 1$),
but this is all straightforward algebra.
\end{proof}

\begin{example}
It holds that 
{\[
\ba
\opeakR_{(3,2,5,2)}  \opeakR_{(4,2)} = \opeakR_{(3,2,5,3,3,2)} &+ 2\cdot\opeakR_{(3,2,5,6,2)}  + \opeakR_{(3,2,5,2,4,2)} \\&+3\beta\cdot \opeakR_{(3,2,5,5,2)} +\beta \cdot\opeakR_{(3,2,5,2,3,2)} + \beta^2\cdot \opeakR_{(3,2,5,4,2)}
\ea
  \]
  }while
$
\opeakR_{(3,2,5,1)}  \opeakR_{(4,2)} =  \opeakR_{(3,2,5,2,3,2)} +2\cdot\opeakR_{(3,2,5,5,2)} +2\beta \cdot\opeakR_{(3,2,5,4,2)} .
 $
 \end{example}

As above, for $n \in \PP$ we set 
\be\label{opeak-Rn}
\opeakR_{n}:=\opeakR_{(n)} = \begin{cases} 2\cdot \tpeakR_n + \beta\cdot \tpeakR_{n-1}&\text{if }n>1 \\
2\cdot \tpeakR_n &\text{if }n=1
\end{cases}
\quand \opeakR_0 := 1.\ee
The following identities suffice to compute coproducts in $\MPeakP$ and $\MPeakQ$:

 \begin{proposition}\label{peak-coproduct-prop}
If $n\in \PP$ then  $ \Delta(\opeakR_n) =
 \sum_{i=0}^{n} \opeakR_i \otimes \opeakR_{n-i}$ and
\[ \Delta(\tpeakR_n) =
1\otimes \tpeakR_{n}   + 
 \sum_{i=1}^{n} \tpeakR_i \otimes \opeakR_{n-i}
=
  \tpeakR_{n}\otimes 1 + 
 \sum_{i=0}^{n-1} \opeakR_i \otimes \tpeakR_{n-i} 
 .\]

\end{proposition}

\begin{proof}
Let $\alpha'$ and $\alpha''$ be peak compositions
and choose
 $u \in \Sm_p$ and $v \in \Sm_q$ with $\cD(u) = \alpha'$ and $\cD(v) =\alpha''$.
 Keeping in mind the product formula in Proposition~\ref{K-prod-prop},
suppose $w \in \Sm_{p+q}$ has $u \leq_\m w \cap [p]$ and $v \uparrow p \leq_\m w \cap (p + [q])$.
Write $m:=\ell(v) = |\alpha''|$.
The only way that  we can have $\PeakSet(w) = \varnothing$ is
if  $\PeakSet(u)=\PeakSet(v) = \varnothing$ and
$w$ is either $\tilde v_1 \cdots \tilde v_i \cdot u\cdot \tilde v_{i+1} \cdots \tilde v_m$ or $\tilde v_1 \cdots \tilde v_{i-1} \cdot u\cdot  \tilde v_{i} \cdots \tilde v_m$ or $\tilde v_1 \cdots \tilde v_i \cdot u\cdot  \tilde v_{i} \cdots \tilde v_m$
where $i \in [m]$ is the index of the   smallest letter of $v$ and $\tilde v_j := v_j+p$.

Thus by Proposition~\ref{K-prod-prop} we have
 $
[\Delta(\tpeakR_n), K^{(\beta)}_{\alpha'} \otimes K^{(\beta)}_{\alpha''}] = 
[\tpeakR_n, K^{(\beta)}_{\alpha'} K^{(\beta)}_{\alpha''}  ] 
 \in \{0,1,2,\beta\}$.
Specifically, the value of the form is zero if $\ell(\alpha') > 1$ or $\ell(\alpha'')>1$
as then $\PeakSet(u)$ or $ \PeakSet(v)$ is nonempty.
If this does not occur, then the value  of the form
is $\ell(\alpha') + \ell(\alpha'')$ when $n= |\alpha'| + |\alpha''|$,
or
$\beta$ when $\ell(\alpha') = \ell(\alpha'')=1$ and $n = |\alpha'| + |\alpha''| + 1$,
or else zero.
We conclude by the definition of $ [\cdot,\cdot ]$ that 
$ \Delta(\tpeakR_n) =
1\otimes \tpeakR_{n} +  \tpeakR_{n}\otimes 1 + 
 2\sum_{i=1}^{n-1} \tpeakR_i \otimes \tpeakR_{n-i} + \beta \sum_{i=1}^{n-2} \tpeakR_i \otimes \tpeakR_{n-i-1}$.
 This identity is equivalent to the displayed equation for $ \Delta(\tpeakR_n)$ via \eqref{opeak-Rn}.
 It follows that 
 \[\Delta(\opeakR_1) = 2\cdot \Delta(\tpeakR_1)  = 1\otimes \opeakR_{n} +  \opeakR_{n}\otimes 1 
\quand \Delta(\opeakR_n) =2\cdot \Delta( \tpeakR_n) + \beta\cdot  \Delta(\tpeakR_{n-1}) \]
when $n>1$.
The expression on the right expands to 
\[
2\otimes \tpeakR_{n} +  \tpeakR_{n}\otimes 2 + \beta\otimes \tpeakR_{n-1} +   \tpeakR_{n-1}\otimes \beta + 
 2\sum_{i=1}^{n-1} \tpeakR_i \otimes \opeakR_{n-i} +  \beta\sum_{i=1}^{n-2} \tpeakR_i \otimes \opeakR_{n-i - 1}
 \]
 and one can check that this is equal to $ \sum_{i=0}^{n} \opeakR_i \otimes \opeakR_{n-i} $.
%
%
%
 \end{proof}
 
There is no $Q$-version of Proposition~\ref{MPeak-free-prop}.
 Over $\ZZ[\beta]$, the set $\{\opeakR_n: n=1,3,5,\dots\}$
generates a proper subalgebra of $\MPeakQ$
which contains $ 2\cdot \opeakR_n$ but not $\opeakR_n$ for even $n \in \PP$.  
This set does freely generate  $\QQ[\beta]\otimes_{\ZZ[\beta]}\MPeakP=\QQ[\beta]\otimes_{\ZZ[\beta]}\MPeakQ$
as a $\QQ[\beta]$-algebra.

The classical peak algebra is also freely generated by a countable set \cite[Thm.~3]{Schocker},
so after an appropriate extension scalars 
it is isomorphic to $\QQ[\beta]\otimes_{\ZZ[\beta]}\MPeakP=\QQ[\beta]\otimes_{\ZZ[\beta]}\MPeakQ$
as an algebra.
By duality, the tensor product  $\QQ[\beta]\otimes_{\ZZ[\beta]}\mQSymQ= \QQ[\beta]\otimes_{\ZZ[\beta]}\mQSymP$ 
is isomorphic as a coalgebra to the completion of the peak quasisymmetric functions
 $\mathbf{\Pi}_\QQ$ from
 \cite{BMSW,Stembridge1997a}.

The (co)algebra isomorphisms that come from these observations
 do not extend to isomorphisms of Hopf algebras.
This is different from the unshifted case \eqref{k-diagram}, 
 where both  $ L^{(\beta)}_{\alpha}$
 and $ L_\alpha := L^{(0)}_{\alpha}$ span $\mQSym$, 
 and we have $\MNSym \cong \NSym$ as Hopf algebras (after an appropriate extension of scalars) \cite[Prop.~8.5]{LamPyl}.
In principle, our shifted objects might still be isomorphic to their $\beta=0$ specializations by some other maps;
determining whether such maps exist is an open problem.
  
\subsection{Shifted symmetric functions}\label{shifted-sym-sect}

Finally, we turn to the shifted analogues of symmetric functions that arise in $K$-theory.
Recall that a partition is strict if its nonzero parts are all distinct.
Choose strict partitions $\mu \subseteq \lambda$ and write $
\SD_{\lambda/\mu} := \{(i,i+j-1) \in \PP\times \PP : \mu_i< j \leq \lambda_i\}$
for the \defn{shifted Young diagram}. We often refer to the positions in this diagram as \defn{boxes}.
A \defn{shifted set-valued tableau} of shape $\lambda/\mu$
is a map $T $ assigning nonempty finite subsets of 
$\{1'<1<2'<2<\dots\}$ to the boxes in $\SD_{\lambda/\mu}$.
We write $(i,j) \in T$ when $(i,j) \in \SD_{\lambda/\mu}$
and let $T_{ij}$ denote the set assigned by $T$ to box $(i,j)$.

A shifted set-valued tableau has \defn{weakly increasing rows and columns} if 
$
\max(T_{ij}) \leq \min(T_{i+1,j})
$
and
$\max(T_{ij}) \leq \min(T_{i,j+1})$
for all relevant $(i,j)\in T$.
A shifted set-valued tableau $T$ with this property is \defn{semistandard} if 
no primed number occurs in multiple boxes of $T$ in the same row and 
no unprimed number occurs in multiple boxes of $T$ in the same column.
For example, 
\[
\ytableausetup{boxsize=0.65cm,aligntableaux=center}
 \begin{ytableau}
\none & \none & 345 & \none\\
\none &  12  & 23' & 7 &  \none\\
\none[\cdot] & \none[\cdot] & 1' & 2'3
\end{ytableau}
\quand
\begin{ytableau}
\none & \none & 3'4 & \none\\
\none &  2' 2  & 3' &  7'& \none\\
\none[\cdot] & \none[\cdot]  & 1'1& 235
\end{ytableau}
\]
are semistandard shifted set-valued tableaux of shape $(4,3,1)/(2)$ drawn in French notation.
Given such a tableau $T$, we let $ |T| := \sum_{(i,j) \in T} |T_{ij}|$
and
$ x^T := \prod_{(i,j) \in T} \prod_{k \in T_{ij}} x_{\lceil k\rceil }$.
Both of our examples have $|T| = 11$ and $x^T = x_1^2 x_2^3 x_3^3 x_4x_5 x_7$.
Also set 
 $|\lambda/\mu| := |\SD_{\lambda/\mu|}$.
The following definitions originate in work of Ikeda and Naruse \cite[\S9]{IkedaNaruse}:

\begin{definition}
Let $\ShSetTab_Q(\lambda/\mu)$ denote the set of all semistandard shifted set-valued tableaux of shape $\lambda/\mu$,
and let $\ShSetTab_P(\lambda/\mu)$ be the subset of such tableaux with no primed numbers in diagonal boxes.
The \defn{$K$-theoretic Schur $P$- and $Q$-functions}   are the formal power series
\[
\bGP_{\lambda/\mu} := \sum_{T \in \ShSetTab_P(\lambda/\mu)} \beta^{|T|-|\lambda/\mu|} x^T
\quand
\bGQ_{\lambda/\mu} := \sum_{T\in \ShSetTab_Q(\lambda/\mu)} \beta^{|T|-|\lambda/\mu|} x^T. 
\]
When $\mu=\emptyset$ is the empty partition
we write 
$\bGP_{\lambda} := \bGP_{\lambda/\emptyset}$ and $\bGQ_{\lambda} := \bGQ_{\lambda/\emptyset}.$
\end{definition}

If  $\deg(\beta) =0$ and $ \deg(x_i)=1$, then  $\bGP_{\lambda/\mu}$ and $\bGQ_{\lambda/\mu}$
have unbounded degree, but their lowest degree terms 
are the Schur $P$- and Schur $Q$-functions $P_{\lambda/\mu}$ and $Q_{\lambda/\mu} = 2^{\ell(\lambda)-\ell(\mu)} P_{\lambda/\mu}$. 
As $\{P_\lambda\}$ and $\{Q_\lambda\}$ are bases for   subalgebras of $\MSym$,
the sets $\{\bGP_\lambda\}$ and $\{\bGQ_\lambda\}$ are linearly independent.

\begin{definition}
Define $\mSymP $ and $\mSymQ $
to be the  linearly compact $\ZZ[\beta]$-modules
with the sets $\{\bGP_\lambda\}$ and $\{\bGQ_\lambda\}$ 
(where $\lambda$ ranges over all strict partitions)
as respective pseudobases.
\end{definition}

If  one sets $\deg(\beta) =-1$ and $ \deg(x_i)=1$
then $\bGP_{\lambda/\mu}$
and
$\bGQ_{\lambda/\mu}$ are homogeneous of degree $|\lambda/\mu|$. 
Both $\mSymP $ and $\mSymQ $ are LC-Hopf subalgebras of $\mSym$ \cite[Thm.~5.11]{LM2019}
and if 
$\mu \subseteq \lambda$ are strict partitions
then
$\bGP_{\lambda/\mu}\in \mSymP$
and
$\bGQ_{\lambda/\mu} \in \mSymQ$ \cite[Cor.~5.13]{LM2019}.

It remains to identify the dual Hopf algebras $\MSymP $ and $\MSymQ $ in \eqref{shk-diagram}.
Continue to assume $\mu \subseteq \lambda$ are strict partitions.
A \defn{shifted reverse plane partition} of shape $\lambda/\mu$
is an assignment of numbers from $\{1'<1<2'<2<\dots\}$ to the boxes in $\SD_{\lambda/\mu}$
such that 
rows and columns are weakly increasing.
If $T$ is a shifted reverse plane partition, then we let
\[\wtmrpp(T) := (a_1+b_1, a_2+b_2,\dots)
\quand
|\wtmrpp(T) | := a_1+b_1+ a_2+b_2 + \dots\]
where
 $a_i$ is the number of distinct columns of $T$ containing $i$
and $b_i$ is the number of distinct rows of $T$ containing  $i'$.
For example, if $\lambda = (4,3,2)$ and $\mu=(2)$ then $T$ could be either of
\[
\ytableausetup{boxsize=0.45cm,aligntableaux=center}
 \begin{ytableau}
\none & \none & 5' & 5'\\
\none &  3  & 3 & 5'\\
\none[\cdot] & \none[\cdot] & 1' & 1'
\end{ytableau}
\quand
\begin{ytableau}
\none & \none & 5 & 5\\
\none &  3'  & 3' & 3\\
\none[\cdot] & \none[\cdot] & 1 & 3
\end{ytableau}
\]
and we would have $\wtmrpp(T) = (1,0,2,0,2)$ and $|\wtmrpp(T)| = 5$ in both cases.

\begin{definition}\label{bgpq-def}
Let $\MRPP_Q(\lambda/\mu)$ be the set of
shifted reverse plane partitions of shape $\lambda/\mu$,
and let  
$\MRPP_P(\lambda/\mu)$ 
be the subset of  $T \in \MRPP_Q(\lambda/\mu)$
whose diagonal entries are all primed.
The  \defn{dual $K$-theoretic Schur $P$- and $Q$-functions}
are
\[
\bgp_{\lambda/\mu} := \sum_{T\in \MRPP_P(\lambda/\mu)} (-\beta)^{|\lambda/\mu| - |\wtmrpp(T)|} x^{\wtmrpp(T)} 
\]
and
\[
\bgq_{\lambda/\mu} := \sum_{T\in \MRPP_Q(\lambda/\mu)} (-\beta)^{|\lambda/\mu| -  |\wtmrpp(T)|} x^{\wtmrpp(T)} .
\]
When $\mu=\emptyset$ is the empty partition
we write 
$\bgp_{\lambda} := \bgp_{\lambda/\emptyset}$ and $\bgq_{\lambda} := \bgq_{\lambda/\emptyset}.$ We will adopt a similar convention for all later notations indexed by skew shapes $\lambda/\mu$.
\end{definition}

By \cite[Thm.~1.4]{LM2022}, $\bgp_{\lambda}$
and
$\bgq_{\lambda}$
are special cases of the \defn{dual universal factorial Schur $P$- and $Q$-functions}
that Nakagawa and Naruse characterize in \cite[Def.~3.2]{NN2018} via a general Cauchy identity.
The skew versions $\bgp_{\lambda/\mu} $ and $\bgq_{\lambda/\mu} $
of these functions were first considered in \cite[\S6]{ChiuMarberg}.

If  we set  $\deg(\beta) = \deg(x_i)=1$, then  $\bgp_{\lambda/\mu}$ and $\bgq_{\lambda/\mu}$ 
are both homogeneous of (bounded) degree $|\lambda/\mu|$.
If instead  $\deg(\beta) = 0$ and $\deg(x_i)=1$, then the
terms of highest degree in $\bgp_{\lambda/\mu}$ and $\bgq_{\lambda/\mu}$
are $P_{\lambda/\mu}$ and $Q_{\lambda/\mu} $, so
$\{\bgp_\lambda\}$ and $\{\bgq_\lambda\}$ are linearly independent over $\ZZ[\beta]$.

\begin{definition}
Define $\MSymP $ and $\MSymQ $
to be the free $\ZZ[\beta]$-modules
with the sets $\{\bgp_\lambda\}$ and $\{\bgq_\lambda\}$ 
(where $\lambda$ ranges over all strict partitions)
as respective bases.
\end{definition}

The functions $\bgp_{\lambda/\mu}$
and
$\bgq_{\lambda/\mu}$ satisfy
the following Cauchy identity.
Let $\mathbf x = (x_1,x_2,\dots)$ and $\mathbf y = (y_1,y_2,\dots)$ be  commuting variables
and set $\overline{x}_i = \frac{-x_i}{1+\beta x_i}$. Then
\be\label{cauchy-eq}
\ds\sum_{\lambda} \bGP_{\lambda}(\mathbf x) \bgq_{\lambda}(\mathbf y) =
\sum_{\lambda} \bGQ_{\lambda}(\mathbf x) \bgp_{\lambda}(\mathbf y)=
 \prod_{i,j\geq 1} \tfrac{1-\overline{x}_i y_j}{1-x_iy_j} 
 \ee
 by \cite[Thm.~1.4]{LM2022} via \cite[Conj. 5.1]{NN2018}.\footnote{\cite[Conj. 5.1]{NN2018} asserts that the power series
 $\bgq_{\lambda}$ and $ \bgp_{\lambda}$ \emph{defined} by \eqref{cauchy-eq} have the generating function formulas in Definition~\ref{bgpq-def},
 and
  \cite[Thm.~1.4]{LM2022} proves this conjecture.}
Since  the right hand side of \eqref{cauchy-eq} is invariant under permutations of the $\bf y$ variables,
the power series $\bgp_{\lambda} $ and $ \bgq_{\lambda}$ are symmetric.
This implies that 
\be\label{gp-gq-xy-eq}
\bgp_{\lambda} ({\bf x},{\bf y}) = \sum_{\mu} \bgp_\mu({\bf x}) \bgp_{\lambda/ \mu} ({\bf y})
\quand
\bgq_{\lambda} ({\bf x},{\bf y}) = \sum_{\mu} \bgq_\mu({\bf x}) \bgq_{\lambda/ \mu} ({\bf y}),
\ee
where $f({\bf x},{\bf y})$ denotes the power series $f(x_1,y_1,x_2,y_2,\dots)$ for $f \in \ZZ[\beta]\llbracket x_1,x_2,\dots\rrbracket $.
The sums here are over all strict partitions $\mu$, setting $\bgp_{\lambda/\mu} = \bgq_{\lambda/\mu}=0$
when $\mu\not\subseteq \lambda$.
Both sides of \eqref{gp-gq-xy-eq} are symmetric under all permutations of the $\bf y$ variables,
so  $\bgp_{\lambda/ \mu} $ and $\bgq_{\lambda/ \mu} $ are also   symmetric.

%
%
%
%

\begin{remark}\label{coproduct-rmk}
The continuous $\ZZ[\beta]$-linear map with $f({\bf x})g({\bf y}) \mapsto 
f \otimes g$
for all $f,g \in \ZZ[x_1,x_2,\dots]$
is a bijection $\ZZ[\beta]\llbracket x_1,y_1,x_2,y_2,\dots\rrbracket  \xrightarrow{\sim} \ZZ[\beta]\llbracket x_1,x_2,\dots\rrbracket \htimes \ZZ[\beta]\llbracket x_1,x_2,\dots\rrbracket $.
Composing this map
with $f \mapsto f({\bf x},{\bf y})$
gives an
operation $\ZZ[\beta]\llbracket x_1,x_2,\dots\rrbracket  \to \ZZ[\beta]\llbracket x_1,x_2,\dots\rrbracket \htimes \ZZ[\beta]\llbracket x_1,x_2,\dots\rrbracket $
which  restricts to the coproduct of $\mSym$ and $\MSym$ \cite[\S2.1]{GrinbergReiner}.
Thus we can rewrite  \eqref{gp-gq-xy-eq} as
\be\label{gp-gq-xy-eq2}
\Delta(\bgp_{\lambda}) = \sum_{\mu} \bgp_\mu \otimes \bgp_{\lambda/ \mu} 
\quand
\Delta(\bgq_{\lambda}) = \sum_{\mu} \bgq_\mu \otimes \bgq_{\lambda/ \mu} 
\ee
where the sums are over all strict partitions.
\end{remark}

\subsection{Finite expansions and duality}\label{pq-sect2}

To identify the algebraic structure of $\MSymP $ and $\MSymQ $ we need a short digression.
Suppose $\lambda\subseteq \nu$ are strict partitions. 
A box $(i,j) \in \SD_\lambda$ is a \defn{removable corner} of $\lambda$ if $\SD_\lambda - \{(i,j)\} = \SD_\mu$
for a strict partition $\mu \subsetneq \lambda$.
Let $\RC(\lambda)$ be the set of all such boxes
and define
\be\label{ss-eq}
\bGP_{\nu\ss \lambda} := \sum_{\substack{\mu \subseteq \lambda \\ \SD_{\lambda/\mu}\subseteq \RC(\lambda)}} \beta^{|\lambda|-|\mu|} \bGP_{\nu/\mu}
\quand
\bGQ_{\nu\ss \lambda} := \sum_{\substack{\mu \subseteq \lambda \\ \SD_{\lambda/\mu}\subseteq \RC(\lambda)}} \beta^{|\lambda|-|\mu|}\bGQ_{\nu/\mu}.
\ee
For strict partitions $\lambda\not\subseteq\nu$ set $\bGP_{\nu\ss \lambda} = \bGQ_{\nu\ss \lambda} :=0$. 
These functions arise in the identities
\be 
\bGP_\nu({\bf x},{\bf y}) = \sum_{\lambda} \bGP_\lambda({\bf x}) \bGP_{\nu\ss \lambda} ({\bf y})
\quand
\bGQ_\nu({\bf x},{\bf y}) = \sum_{\lambda} \bGQ_\lambda({\bf x}) \bGQ_{\nu\ss \lambda} ({\bf y})
\ee
which by Remark~\ref{coproduct-rmk} can be restated as the coproduct formulas
\be \label{Delta-eq}
\Delta(\bGP_\nu) = \sum_{\lambda} \bGP_\lambda \otimes \bGP_{\nu\ss \lambda} 
\quand
\Delta(\bGQ_\nu) = \sum_{\lambda} \bGQ_\lambda\otimes \bGQ_{\nu\ss \lambda}.
\ee
The sums here are over all strict partitions $\lambda$, but the terms indexed by $\lambda\not\subseteq \nu$ are all zero.


Because  $\mSymP $ and $\mSymQ $ are LC-Hopf algebras, and because 
the pseudobases $\{\bGP_\lambda\}$ and $\{\bGQ_\lambda\}$
 consist of homogeneous elements if we set $\deg(\beta) = -1$, there are unique integers 
$a_{\lambda\mu}^\nu, b_{\lambda\mu}^\nu,
\widehat a_{\lambda\mu}^\nu, \widehat b_{\lambda\mu}^\nu \in \ZZ$ indexed by strict partitions $\lambda$, $\mu$, $\nu$ such that 
\be\label{ab-eq}
\ba
\bGP_\lambda \bGP_\mu  &= \sum_\nu a_{\lambda\mu}^\nu \beta^{|\nu|-|\lambda|-|\mu|} \bGP_\nu\\
\bGQ_\lambda \bGQ_\mu  &= \sum_\nu b_{\lambda\mu}^\nu \beta^{|\nu|-|\lambda|-|\mu|} \bGQ_\nu
\ea
\quand
\ba
\bGP_{\nu\ss\lambda}&=\sum_\mu  \widehat b_{\lambda\mu}^\nu \beta^{|\lambda| + |\mu| - |\nu|} \bGP_\mu
\\
\bGQ_{\nu\ss\lambda}&=\sum_\mu  \widehat a_{\lambda\mu}^\nu \beta^{|\lambda| + |\mu| - |\nu|} \bGQ_\mu.
\ea
\ee
The following result is a special case of \cite[Prop.~3.2]{NN2018}
since $\bgp_\lambda$ and $\bgq_\lambda$ are special cases of \cite[Def.~3.2]{NN2018}.
We outline a self-contained proof for completeness.

\begin{proposition}[{\cite[Prop.~3.2]{NN2018}}] For all strict partitions $\lambda$, $\mu$, $\nu$ it holds that
\label{abcd-prop}
\[
\ba
\bgp_\lambda \bgp_\mu  &= \sum_\nu \widehat a_{\lambda\mu}^\nu \beta^{|\lambda|+|\mu|-|\nu|} \bgp_\nu
\\
\bgq_\lambda \bgq_\mu  &= \sum_\nu \widehat b_{\lambda\mu}^\nu \beta^{|\lambda| + |\mu| - |\nu|} \bgq_\nu
\ea
\quand
\ba
\bgp_{\nu/\lambda}&=\sum_\mu   b_{\lambda\mu}^\nu \beta^{ |\nu|-|\lambda| - |\mu|} \bgp_\mu
\\
\bgq_{\nu/\lambda}&=\sum_\mu   a_{\lambda\mu}^\nu \beta^{ |\nu|-|\lambda| - |\mu|} \bgq_\mu.
\ea
\]
\end{proposition}

\begin{proof}
One can derive these identities
from \eqref{cauchy-eq}
by introducing a third sequence of variables ${\mathbf z} = (z_1,z_2,\dots)$   
and then extracting coefficients.
For example, we have
$
\sum_\nu \bGQ_\nu({\bf x},{\bf y})  \bgp_\nu({\bf z})  =  \prod_{i,j} \tfrac{1-\overline{x}_i z_j}{1-x_iz_j}    \prod_{i,j} \tfrac{1-\overline{y}_i z_j}{1-y_iz_j} $.
The left side is
$\sum_{\lambda,\mu,\nu}  \widehat a_{\lambda\mu}^\nu \beta^{|\lambda| + |\mu| - |\nu|} \bGQ_\lambda({\bf x}) \bGQ_\mu({\bf y}) \bgp_\nu({\bf z})$
while the right side is 
$ \sum_{\lambda,\mu} \bGQ_\lambda({\bf x})  \bGQ_\mu({\bf y})  \bgp_\lambda ({\bf z}) \bgp_\mu({\bf z})$,
which leads to the first formula.
\end{proof}

Let $\ell(\lambda)$ be the number of parts in a partition $\lambda$.
Given strict partitions $\mu \subseteq \lambda$, define $\columns(\lambda/\mu) = |\{ j : (i,j) \in \SD_{\lambda/\mu}\}|$ to be the number of columns occupied by   $\SD_{\lambda/\mu}$.
A subset of $\PP\times \PP$ is a \defn{vertical strip} if it contains at most one position in each row.
Then it holds by \cite[Thm.~1.1]{ChiuMarberg}
that 
 \be\label{GQ-to-GP-eq}
 \bGQ_\mu =  \sum_{\lambda}  2^{\ell(\mu)}  (-1)^{\columns(\lambda/\mu)}  (-\beta/2)^{|\lambda/\mu|}  \bGP_\lambda
 \ee
 and by  \cite[Cor.~6.2]{ChiuMarberg} that
 \be\label{gq-to-gp-eq}
 \bgq_\lambda =   \sum_{\mu} 2^{\ell(\mu)}  (-1)^{\columns(\lambda/\mu)}  (-\beta/2)^{|\lambda/\mu|}  \bgp_\mu
\ee
where both sums are over strict partitions $\lambda\supseteq \mu$ 
with $\ell(\lambda) =\ell(\mu)$ such that $\SD_{\lambda/\mu}$ is a vertical strip.
For example
$\bGQ_{(3,2)} = 4 \bGP_{(3,2)} + 2\beta   \bGP_{(4,2)} - \beta^2  \bGP_{(4,3)}$
and 
$\bgq_{(3,2)} = 4 \bgp_{(3,2)} + 2\beta  \bgp_{(3,1)} -\beta^2  \bgp_{(2,1)}.$

We use the following notation from \cite{Iwao} just in the next result.
Let $\mSymPQ$ 
be the linearly compact $\QQ(\beta)$-module with $\{ \bGP_\lambda\}$  as a pseudobasis
(where $\lambda$ ranges over all strict partitions)
and let
 $\MSymPQ$ be the free $\QQ(\beta)$-module with 
 $\{ \bgp_\lambda\}$  as a basis.
The identities  \eqref{GQ-to-GP-eq} and \eqref{gq-to-gp-eq}
imply that $\{ \bGQ_\lambda\}$ is another pseudobasis for $\mSymPQ$ 
while  $\{ \bgq_\lambda\}$ is another basis for  $\MSymPQ$.

\begin{proposition}\label{sh-form-prop}
There is a unique bilinear form $[\cdot,\cdot] :  \MSymPQ \times \mSymPQ\to \QQ(\beta)$, continuous in the second coordinate,
with
$[ \bgp_\lambda, \bGQ_\mu] = [ \bgq_\lambda, \bGP_\mu] = \delta_{\lambda\mu}$
for all strict partitions  $\lambda$ and $\mu$. 
\end{proposition}

\begin{proof}
Let $(\cdot,\cdot)$ be the bilinear form, continuous in the second coordinate,
with $(\bgp_\lambda,\bGP_\mu) = \delta_{\lambda\mu}$.
Write $\phi$ and $\psi$ for the (continuous) linear maps with $\phi (\bgq_\lambda) = \bgp_\lambda$ and $\psi (\bGQ_\mu) = \bGP_\mu$.
The identities \eqref{GQ-to-GP-eq} and \eqref{gq-to-gp-eq} imply that 
$(\phi(f), g) = (f, \psi(g))$ 
for all $f \in  \MSymPQ $ and $g\in \mSymPQ$.
Then the form $[f,g] :=(\phi(f), g) = (f, \psi(g))$ has the desired properties.
\end{proof}

\begin{remark*}
The module $\mSymPQ$ coincides with 
the ring considered in \cite[\S5.2]{Iwao}, which Iwao defines  by a certain \defn{$K$-theoretic $Q$-cancellation property};
\cite[Thm.~3.1]{IkedaNaruse} shows that the infinite linear span of the $\bGP_\lambda$'s is characterized by the same property.
Comparing the Cauchy identity \eqref{cauchy-eq}  with the one in \cite[\S8.2]{Iwao}
  shows that $\MSymPQ$ similarly coincides with the ring defined in \cite[\S8.1]{Iwao},
  and that the form in Proposition~\ref{sh-form-prop} is equal to the one in \cite[Eq. (30)]{Iwao}.
\end{remark*}

Putting everything together leads to this theorem:

\begin{theorem}\label{gg-thm}
 Both $\MSymP $ and $\MSymQ $ are Hopf subalgebras of $\MSym$.
In particular, $\MSymP$ (respectively, $\MSymQ$) is the Hopf algebra dual to  $\mSymQ$ (respectively, $\mSymP$)
via  $[\cdot,\cdot]$.
\end{theorem}

\begin{proof}
It is clear from \eqref{gp-gq-xy-eq2} and Proposition~\ref{abcd-prop}
that $\MSymP $ and $\MSymQ $ are the sub-bialgebras of $\MSym$
dual to 
 $\mSymQ$ and $\mSymP$
via   $[\cdot,\cdot]$.
The fact that these bialgebras are preserved by the antipode of $\MSym$
follows by duality as $\mSymQ$ and $\mSymP$ are LC-Hopf subalgebras of $\mSym$.
\end{proof}

As   mentioned in Section~\ref{sym-sect},
the Hopf algebras $\MSym$ and $\mSym$ have another pair of dual bases for $\langle\cdot,\cdot\rangle$ besides the Schur functions,
given by the (dual) stable Grothendieck polynomials $\{ g^{(\beta)}_\lambda\}$ and $\{ G^{(\beta)}_\lambda\}$.
We need to quote two results involving these functions. First, one has 
 \be\label{gpn-eq}
 \gp_{(n)} = \sum_{i=1}^n g^{(\beta)}_{(i,1^{n-i})}
 \ee for all positive integers $n$ by \cite[Prop.~5.3]{NN2018}; see the proof of \cite[Prop.~7.5]{ChiuMarberg} for another derivation.
Second, if $\lambda$ is a partition with $k$   parts 
and $\ttheta: \mQSym \to \mQSymQ
$ is the map \eqref{ttheta-eq}
then 
\be\ttheta(G^{(\beta)}_\lambda) = \bGQ_{(\lambda+\delta)/\delta}\quad\text{for $\delta := (k-1,\dots,2,1,0)$, by \cite[\S4.6]{LM2019}.}
\ee 
Taking $k=1$ gives $\ttheta(G^{(\beta)}_{(n)}) = \bGQ_{(n)}$ for all $n>0$.
Thus $\ttheta$ restricts to a map $\mSym \to \mSymQ$,
which is surjective by \cite[Cor.~5.17]{LM2019}.
The following result reduces to \cite[(A.9)]{Stembridge1997a} when $\beta=0$:

 \begin{theorem}\label{ttheta-thm}
If $f \in \MSymPQ$ and $g \in \mSym$ then $\left[f,\ttheta(g)\right] = \langle f,g\rangle$.
\end{theorem}

\begin{proof}
Write 
$\bgp_n := \bgp_{(n)}$
and $\bgq_n := \bgq_{(n)}$ for $n \in \PP$.
Define $G^{(\beta)}_n$ and $\bGQ_n$ analogously.
For compositions $\alpha = (\alpha_1,\alpha_2,\dots,\alpha_k)$
let $G(\alpha) :=\prod_{i\in[k]} G^{(\beta)}_{\alpha_i}$,
 $\GQ(\alpha):=\ttheta(G(\alpha)) = \prod_{i\in[k]} \bGQ_{\alpha_i}$,
 and
  $\gp(\alpha):=\prod_{i\in[k]} \bgp_{\alpha_i}$.
  The Pieri rule in \cite[Thm.~3.4]{Lenart2000} implies that every element of $\mSym$ is possibly infinite $\ZZ[\beta]$-linear combination of $G(\alpha)$'s.
  The analogous Pieri rule in \cite[Prop.~2.7]{LM2022} with \eqref{gq-to-gp-eq}
  implies
that every element of $\MSymPQ$ is a finite $\QQ(\beta)$-linear combination of $\gp(\alpha)$'s.
  Thus we just need to check that 
  $[\gp(\alpha), \GQ(\gamma)] = \langle \gp(\alpha),G(\gamma)\rangle$
  for all compositions $\alpha$ and $\gamma$.

To show this, let $k=\ell(\alpha)$ and $l=\ell(\gamma)$ be the lengths of two arbitrary compositions.
Recall that we denote iterated coproducts by $\Delta^{(k)} := (1 \otimes \Delta^{(k-1)})\circ \Delta = ( \Delta^{(k-1)}\otimes 1 )\circ \Delta$ where $\Delta^{(1)}:=\Delta$. 
Since $[\cdot,\cdot]: \MSymP\times \mSymQ \to \ZZ[\beta]$
and $\langle \cdot,\cdot\rangle : \MSym \times \mSym\to\ZZ[\beta]$
induce Hopf algebra dualities and since coproducts in Hopf algebras are algebra morphisms,
we have
\[
[\gp(\alpha), \GQ(\gamma)] 
=
 \left[\Delta^{(l-1)}(\gp(\alpha)), \bigotimes_{j \in[l]} \bGQ_{\gamma_j} \right] 
 =
 \left[\bigotimes_{i \in[k]}  \Delta^{(l-1)}( \bgp_{\alpha_i}) ,\bigotimes_{j \in[l]} \Delta^{(k-1)}( \bGQ_{\gamma_j}) \right] 
\]
and similarly
\[
\langle\gp(\alpha), G(\gamma)\rangle
=
 \left\langle\Delta^{(l-1)}(\gp(\alpha)), \bigotimes_{j \in[l]} G^{(\beta)}_{\gamma_j} \right\rangle 
 =
 \left\langle\bigotimes_{i \in[k]}  \Delta^{(l-1)}( \bgp_{\alpha_i}) ,\bigotimes_{j \in[l]} \Delta^{(k-1)}( G^{(\beta)}_{\gamma_j}) \right\rangle
\]
where   we appropriately reorder the $kl$ tensor factors in $\bigotimes_{i \in[k]}  \Delta^{(l-1)}( \bgp_{\alpha_i})$
to evaluate the two rightmost expressions.

It is clear from \eqref{gp-gq-xy-eq2} 
that $\Delta(\bgp_n) =\sum_{i=0}^n \bgp_i\otimes \bgp_{n-i}$
and from  \eqref{Delta-eq} that
$\Delta(\bGQ_n) =\sum_{i=0}^n \bGQ_i\otimes \bGQ_{n-i} +\beta \sum_{i=1}^n \bGQ_i \otimes \bGQ_{n+1-i}$
where $\bgp_0 = \bGQ_0 :=1$.
It follows similarly from the well-known set-valued tableau generating function for $G_\lambda$ (see, e.g., \cite[Eq.\ (3.9)]{LM2019})
that $\Delta(G^{(\beta)}_n) =\sum_{i=0}^n G^{(\beta)}_i\otimes G^{(\beta)}_{n-i} +\beta \sum_{i=1}^n G^{(\beta)}_i \otimes G^{(\beta)}_{n+1-i}$.
We deduce from these formulas that the desired equality   $[\gp(\alpha), \GQ(\gamma)] = \langle \gp(\alpha),G(\gamma)\rangle$
will hold if we can just show that $[\bgp_m, \bGQ_n] =\delta_{mn}= \langle \bgp_m,G^{(\beta)}_n\rangle$ for all $m,n \in \NN$.
This simpler identity is immediate from \eqref{gpn-eq}.
\end{proof}

Recall the elements $\tpeakR_n \in \MPeakP$
and $\opeakR_n \in \MPeakQ$ for $n \in \PP$
 from Section~\ref{MPeak-sect}.

\begin{theorem}\label{last-adj-thm}
The map $\MPeakP \to \MSymP$
adjoint to $\mSymQ\hookrightarrow \mQSymQ$
relative to the forms $[\cdot,\cdot]$ in Theorems~\ref{mpeakp-thm} and \ref{gg-thm}
is the unique algebra morphism with $\tpeakR_n \mapsto \bgp_{n} $ and $\opeakR_n \mapsto \bgq_{n} $. 
This morphism restricts to the map $\MPeakQ \to \MSymQ$ adjoint to $\mSymP\hookrightarrow \mQSymP$.
\end{theorem}

\begin{proof}
We first claim that $[ \tpeakR_n, \bGQ_\lambda]= \delta_{(n),\lambda}$ for all $n \in \NN$ and strict partitions $\lambda$.
 This follows from the discussion in \cite[\S4.6]{LM2019}
 which gives the $K^{(\beta)}_\alpha$-decomposition of $\bGQ_\lambda$.
 In detail, define a \defn{standard shifted set-valued tableau} of shape $\lambda$
 to be a semistandard shifted set-valued tableau of shape $\lambda$
 whose entries are pairwise disjoint nonempty sets, never containing any consecutive integers,
 with union $\{1,2,\dots,N\}$ for some $N \geq |\lambda|$.
 Suppose $T$ is such a tableau and set $|T| := N$. 
 The \defn{peak set} of  $T$ is the set $\PeakSet(T)$  of integers $1<i<N$ such that $i$ appears in a column of $T$ strictly after $i-1$
 and in a row of $T$ strictly before $i+1$. Then
$ \bGQ_\lambda = \sum_\alpha k^\alpha_{\lambda} \cdot\beta^{|\alpha|-|\lambda|}\cdot  K^{(\beta)}_\alpha$
where the sum is over all peak compositions of $\alpha$ and 
$ k^\alpha_{\lambda}$ is the number of standard shifted set-valued tableaux $T$ with $|T|=|\alpha|$ and $\PeakSet(T) = I(\alpha)$
\cite[Eq.\ (4.14)]{LM2019}.
Equivalently, $[ \tpeakR_\alpha, \bGQ_\lambda]=k^\alpha_{\lambda}\cdot \beta^{|\alpha|-|\lambda|}$.

Now observe that if $\ell(\lambda)>1$, then every standard shifted set-valued tableau $T$ 
of shape $\lambda$ has a nonempty peak set, since 
if $i+1 $ is the smallest number in box $(2,2)$ of $T$ then $i \in \PeakSet(T)$.
On the other hand, if $\ell(\lambda)\leq 1$ then there is exactly one standard shifted set-valued tableau $T$ 
of shape $\lambda$, and this tableau has $|T| = |\lambda|$ and $\PeakSet(T) = \varnothing$.
Thus if $\alpha= (n)$ for some $n \in \NN$
so that $I(\alpha) = \varnothing$, 
then we have $k^{\alpha}_\lambda = \delta_{(n),\lambda}$ and $[ \tpeakR_n, \bGQ_\lambda]= \delta_{(n),\lambda}$ as desired.

Our claim shows that $[ \tpeakR_n, \bGQ_\lambda] = [ \bgp_n, \bGQ_\lambda] $ 
so the adjoint map $\MPeakP \to \MSymP$
must send $\tpeakR_n \mapsto \bgp_n$ for all $n \in \NN$.
There is at most one algebra morphism with this property since
 $\{\tpeakR_n : n=1,3,5,\dots\}$ freely generates $\MPeakP$ by Proposition~\ref{MPeak-free-prop}.
The adjoint map also sends $\opeakR_n \mapsto \bgq_{n} $ 
since we have $\opeakR_1 = 2 \cdot \tpeakR_1$ and $\bgq_1 = 2\cdot\bgp_1$
as well as 
$\opeakR_n = 2 \cdot\tpeakR_n + \beta\cdot \tpeakR_{n-1}$ and $\bgq_n = 2\cdot \bgp_n + \beta \cdot\bgp_{n-1}$ for $n>1$
by \eqref{opeak-Rn} and \eqref{gq-to-gp-eq}.
\end{proof}

The maps in Theorem~\ref{last-adj-thm} are shifted analogues of the morphism $\MNSym \to \MSym$ in \eqref{k-diagram}.
As noted earlier, Lam and Pylyavskyy  
show
that the image of $\tR_{\alpha}$ under this map is a specific dual stable Grothendieck polynomial $g^{(\beta)}_{\lambda/\mu}$ \cite[Thm.~9.13]{LamPyl}.
We do not know if this result has a shifted version.
When $\alpha$ is a peak composition,
the images of $\tpeakR_\alpha$ and $\opeakR_\alpha$ under the adjoint maps $\MPeakP \to \MSymP$
and $\MPeakQ \to \MSymQ$
typically are not of the form $\bgp_{\lambda/\mu}$ or $\bgq_{\lambda/\mu}$.
 
\subsection{Antipode formulas}\label{antipode-sect}

We gave the general definition of 
the antipode for a (LC-)Hopf algebra in Section~\ref{prelim-sect}.
Here 
we
describe some specific antipode formulas for the objects in \eqref{shk-diagram}.

If $\alpha$ is a finite sequence then we write $\alpha^\r$ for its reversal.
Given $\alpha \vDash n$,
let $\alpha^\c$ be the unique composition of $n$ with $I(\alpha^\c) =[n - 1] \setminus I(\alpha)$,
and define $\alpha^\t := (\alpha^\c)^\r = (\alpha^\r)^\c$.
For example, 
we have $(3,2)^\r = (2,3)$, $(3,2)^\c = (1,1,2,1)$, and
$(3,2)^\t = (1,2,1,1)$.

Recall from Remark~\ref{beta-rmk} that the homogeneous functions
$ L_\alpha := L^{(0)}_\alpha = \sum_{{\gamma\vDash |\alpha|, I(\gamma) \supseteq I(\alpha)}} M_{\gamma}$
form another basis for $\QSym$. 
Write $\omega : \QSym \to \QSym$
for the linear map with
$\omega(L_\alpha) := L_{\alpha^\t}$.
This map is a Hopf algebra automorphism which preserves $\Sym$,
acting on Schur functions as $\omega(s_\lambda)  =s_{\lambda^\top}$
where $\lambda^\top$ is the transpose of a partition $\lambda$.
The antipode of $\QSym$ is the linear map $\antipode$
with 
\be\label{qsym-antipode-eq}
\antipode(L_\alpha) = (-1)^{|\alpha|} L_{\alpha^\t} = (-1)^{|\alpha|} \omega(L_{\alpha})
\quand \antipode(s_\lambda) = (-1)^{|\lambda|} s_{\lambda^\top} =  (-1)^{|\lambda|} \omega(s_{\lambda})
\ee
for all compositions $\alpha$ and partitions $\lambda$  \cite[\S3.6]{LMW}.
We can extend $\omega$ to a continuous automorphism of $\mQSym$.
The antipode of $\mQSym$ is the continuous extension of the antipode of $\QSym$.

A \defn{multiset} 
is a set allowing repeated elements.
%
%
%
%
Given a peak composition $\alpha = (\alpha_1,\alpha_2,\dots,\alpha_k)$
let $\alpha^\flat := (\alpha_k+1, \alpha_{k-1},\dots,\alpha_2,\alpha_1-1)$ when $k>1$
and set $\alpha^\flat := \alpha$ if $k\leq 1$.
The following statement is equivalent to identities in \cite{LM2019},
and reduces to \cite[Prop.~3.5]{Stembridge1997a} when $\beta=0$.

\begin{proposition}[{\cite[Prop 6.5 and Obs. 6.9]{LM2019}}]
If $\alpha$ is a peak composition then 
\[
\antipode\(K^{(\beta)}_{\alpha^\flat}\)   = \sum_S
(-\beta)^{|S| - N} x^S
\quand \antipode\(\bK^{(\beta)}_{\alpha^\flat}\)   = \sum_S
(-\beta)^{|S| - N} x^S
\]
where both sums 
are over $N$-tuples $S=(S_1  \preceq \cdots \preceq S_N)$
satisfying the same respective conditions \eqref{peak-def-eq1} and \eqref{peak-def-eq2}
as in Definition~\ref{peak-def},
but with each $S_i$ a finite nonempty multi-subset of $\MM$.
\end{proposition}

%


%

\begin{proposition}
The antipode $\antipode$ of $\MPeakQ\supset \MPeakP$ is the algebra anti-automorphism with 
\ben

\item[(a)] $\ds\antipode\(\opeakR_{n}\) = (-1)^n\sum_{k\in[n]} \tbinom{n-1}{k-1} \cdot \beta^{n-k} \cdot\opeakR_{k}$   for all $n\in \PP$, and

\item[(b)] $\ds\antipode\(\tpeakR_{n}\) = (-1)^n\sum_{k\in[n]}  \tbinom{n}{k} \cdot \beta^{n-k} \cdot\tpeakR_{k}$  for all $n\in \PP$.

\een
\end{proposition}

\begin{proof}
The antipode of any Hopf algebra is an anti-endomorphism \cite[Prop.~1.4.10]{GrinbergReiner}
and is invertible when the Hopf algebra is cocommutative \cite[Rem.~1.4.13]{GrinbergReiner}.
One has $\nabla \circ (\antipode\otimes \id) \circ \Delta = \iota\circ\epsilon$ by definition,
so by Proposition~\ref{peak-coproduct-prop} we deduce that
$
\antipode(\opeakR_n) = -\opeakR_n - \sum_{m=1}^{n-1}\antipode(\opeakR_m) \opeakR_{n-m}
$ for all $n \in \PP$.
Checking that the formula in part (a) satisfies this recurrence,
using
Proposition~\ref{opeak-prod-prop},
is a somewhat involved but completely elementary computation, requiring only sum manipulations
and
the identity $\binom{m}{k} = \binom{m-1}{k-1} + \binom{m-1}{k}$.
Part (b) follows from part (a) by \eqref{opeak-Rn},
since one can similarly check that the formula for $\antipode(\tpeakR_{n}) $
is the unique solution to the recurrence with $\antipode(\opeakR_{1})=2\cdot\antipode(\tpeakR_{1}) $
and $\antipode(\opeakR_{n})=2\cdot \antipode(\tpeakR_{n}) + \beta\cdot \antipode(\tpeakR_{n-1}) $ for $n>1$.
    \end{proof}

Patrias \cite[Thms.~33 and 35]{Patrias} computes explicit antipode formulas for $\mQSym$ and $\MNSym$ in the (pseudo)bases $\{ L^{(\beta)}_\alpha\}$ and $\{ \tR_\alpha\}$,
using slightly different notation.\footnote{To convert 
the functions $\tilde L_\alpha$ and $\tilde R_\alpha$
 in \cite{Patrias} to our notation,
set
$\tilde L_\alpha = \beta^{|\alpha|} L^{(\beta)}_\alpha$
and
$\beta^{|\alpha|}\tilde R_\alpha =  \tR_\alpha$.
} We have only given partial formulas for the antipodes of the shifted analogues of these Hopf algebras. 
It may be possible to extend our
results along the lines of \cite{Patrias}.

Now we turn to the shifted versions of stable Grothendieck polynomials.
Choose strict partitions $\mu \subseteq \lambda$.
A \defn{shifted multiset-valued tableau} of shape $\lambda/\mu$
is a map $T $ assigning nonempty finite multi-subsets of 
$\MM$ to the boxes in $\SD_{\lambda/\mu}$.
We write $T_{ij}$ to  denote the multiset assigned by $T$ to position $(i,j)$.
The definition of a \defn{semistandard} shifted multiset-valued tableau
is identical to the set-valued case.

Let $\ShMSetTab_Q(\lambda/\mu)$ denote the set of all semistandard shifted multiset-valued tableaux of shape $\lambda/\mu$,
and let $\ShMSetTab_P(\lambda/\mu)$ be the subset of such tableaux with no primed numbers appearing in diagonal positions.
Then define
\be
\bJP_{\lambda/\mu} := \sum_{T \in \ShMSetTab_P(\lambda/\mu)} (-\beta)^{|T|-|\lambda/\mu|} x^T
\quand
\bJQ_{\lambda/\mu} := \sum_{T\in \ShMSetTab_Q(\lambda/\mu)} (-\beta)^{|T|-|\lambda/\mu|} x^T 
\ee
where as usual  $ |T| := \sum_{(i,j) \in \SD_{\lambda/\mu}} |T_{ij}|$ and
$ x^T := \prod_{(i,j) \in \SD_{\lambda/\mu}} \prod_{k \in T_{ij}} x_{\lceil k\rceil }$.
The functions $ \bJP_{\lambda/\mu}$  become
the \defn{weak shifted stable Grothendieck polynomials} from \cite[\S3]{HKPWZZ}
when $\beta=-1$ and $\mu=\emptyset$.

\begin{proposition}[\cite{LM2019}]\label{GP-GQ-S-prop}
If $\mu\subseteq\lambda$ are strict partitions then 
\[
\antipode\(\bGP_{\lambda/\mu}\)  = (-1)^{|\lambda/\mu|} \bJP_{\lambda/\mu}
\quand
\antipode\(\bGQ_{\lambda/\mu}\)  = (-1)^{|\lambda/\mu|} \bJQ_{\lambda/\mu}.
\]
\end{proposition}


\begin{proof}
This holds since
$\omega\(\bGP_{\lambda/\mu}\)  = \JP^{(-\beta)}_{\lambda/\mu}
$
and
$
\omega\(\bGQ_{\lambda/\mu}\)  =  \JQ^{(-\beta)}_{\lambda/\mu}$ by \cite[Cor.~6.6]{LM2019}.
\end{proof}

A partition of a set $S$ is a set $\Pi$ of disjoint nonempty blocks $B \subseteq S$
with $S = \bigsqcup_{B \in \Pi} B$.
Choose strict partitions $\mu \subseteq \lambda$.
A \defn{semistandard shifted bar tableau} of shape $\lambda/\mu$
is a pair $T=(V,\Pi)$, where $V$ is a semistandard shifted tableau\footnote{That is, a semistandard shifted set-valued tableau whose entries are all sets with exactly one element.
} of shape $\lambda/\mu$
and $\Pi$ is a partition of   $\SD_{\lambda/\mu}$ into subsets of adjacent positions 
 containing the same entry in $V$. One might draw this as a picture like
\be\label{colors-eq}
\begin{young}[13pt][c] 
 , & ]=![cyan!75]2 & ==![cyan!75]  & =]![cyan!75]  \ynobottom & ![pink!75]3'\ynobottom  \\ 
== \cdot \ynotop \ynobottom & \cdot \ynobottom & ]=]![red!75] 1  & ]=]![yellow!75] 1 & ]=]![pink!75] \ynotop & ![magenta!75]3 
\end{young}
\quad
\text{to represent}
\quad
 (V,\Pi) = \(\hspace{0.5mm} \begin{ytableau}
 \none &  2 &  2 &  2 &  3'\\
\none[\cdot]   & \none[\cdot] &  1 &   1 &  3' &  3
 \end{ytableau}\ , 
\begin{young}[13pt][c] 
 , & ]=  & ==  & =]  \ynobottom & \ynobottom  \\ 
== \cdot \ynotop \ynobottom & \cdot \ynobottom & ]=]   & ]=]  & ]=]  \ynotop &  
\end{young}\hspace{0.5mm}
\)
\ee
 when $\lambda = (6,4)$ and $\mu=(2)$.
Let
$ |T| := |\Pi|$ and $ x^T := \prod_{i\geq 1} x_i^{b_i}$
where $b_i$ is the number of blocks in $\Pi$ containing $i$ or $i'$.
Our example \eqref{colors-eq} has   $|T| = 5$ and $x^T = x_1^2 x_2 x_3^2$.

Let $\ShBTQ(\lambda/\mu)$ be the set of semistandard shifted bar tableaux of shape $\lambda/\mu$
and let $\ShBTP(\lambda/\mu)$
be the subset of such tableaux with no primed entries in diagonal positions.
Then define
\[
 \bjp_{\lambda/\mu} := \sum_{T\in \ShBTP(\lambda/\mu)} \beta^{|\lambda/\mu|-|T|} x^{T} 
 \quand
\bjq_{\lambda/\mu} := \sum_{T\in \ShBTQ(\lambda/\mu)} \beta^{|\lambda/\mu| -|T|} 
x^{T} .
\]
These generating functions were first considered as part of some conjectural formulas in \cite[\S7]{ChiuMarberg}.

\begin{proposition}[\cite{LM2022}]
If $\mu\subseteq\lambda$ are strict partitions then 
\[
\antipode\(\bgp_{\lambda/\mu}\)  = (-1)^{|\lambda/\mu|} \bjp_{\lambda/\mu}
\quand
\antipode\(\bgq_{\lambda/\mu}\)  = (-1)^{|\lambda/\mu|} \bjq_{\lambda/\mu}.
\]
\end{proposition}


\begin{proof}
This holds since
$\omega\(\bgp_{\lambda/\mu}\)  = \jp^{(-\beta)}_{\lambda/\mu}
$
and
$
\omega\(\bgq_{\lambda/\mu}\)  =  \jq^{(-\beta)}_{\lambda/\mu}$ by \cite[Thms.~1.4 and 1.5]{LM2022}.
\end{proof}

There are similar formulas  in the unshifted case
 for 
$\antipode\(G^{(\beta)}_{\lambda/\mu}\)$ 
and $\antipode\(g^{(\beta)}_{\lambda/\mu}\)$;
see \cite[\S8]{Patrias}.

\begin{corollary}
For all strict partitions $\mu\subseteq\lambda$ one has
\[\bjp_{\lambda/\mu} \in \MSymP,
\quad\bjq_{\lambda/\mu} \in \MSymQ,
\quad\bJP_{\lambda/\mu} \in \mSymP,
\quand
\bJQ_{\lambda/\mu} \in \mSymQ.\]
Moreover, the form in Proposition~\ref{sh-form-prop}
has
$ [ \bjp_\lambda, \bJQ_\mu] = [ \bjq_\lambda, \bJP_\mu] = \delta_{\lambda\mu}$
for all $\lambda$, $\mu$. 

\end{corollary}

\begin{proof}
The containments hold since $\MSymP$, $\MSymQ$, $\mSymP$, and $\mSymQ$
are all    closed under $\antipode$.
 The antipode of any commutative Hopf algebra is an involution \cite[Cor.~1.4.12]{GrinbergReiner}
 so Theorem~\ref{gg-thm}
implies $(-1)^{|\lambda|+|\mu|} [ \bjp_\lambda, \bJQ_\mu]  =  [ \antipode (\bgp_\lambda), \antipode(\bGQ_\mu)]
= [  \antipode^2(\bgp_\lambda), \bGQ_\mu] =   [  \bgp_\lambda, \bGQ_\mu]
= \delta_{\lambda\mu}$.
One derives the identity
 $[ \bjq_\lambda, \bJP_\mu] = \delta_{\lambda\mu}$ similarly.
\end{proof}

Let $\mathbf{GP} \subsetneq \mSymP$ and $\mathbf{GQ}  \subsetneq \mSymQ$
denote the proper $\ZZ[\beta]$-submodules
with $\{ \bGP_\lambda\}$ and $\{\bGQ_\lambda\}$ as bases, rather than pseudobases, where $\lambda$ ranges over all strict partitions.

\begin{proposition}\label{bialgebras-thm}
Both $\mathbf{GP} $ and $\mathbf{GQ}$
are sub-bialgebras of $\mSym$ but not Hopf algebras.
\end{proposition}

\begin{proof}
We already know that $\{ \bGP_\lambda\}$ and $\{\bGQ_\lambda\}$ are pseudobases for LC-Hopf subalgebras of $\mSym$.
This result makes three nontrivial additional claims.
First, the products  $\bGP_\lambda \bGP_\mu$ and $\bGQ_\lambda \bGQ_\mu$
always expand as finite linear combinations of $\bGP$- and $\bGQ$-functions.
For the $\bGP$-functions, this was first shown in \cite{CTY}; for other proofs, see \cite[\S4]{HKPWZZ}, 
\cite[\S1.2]{Marberg2021}, or \cite[\S8]{PechenikYong}.
For the $\bGQ$-functions, the desired finiteness property is \cite[Thm 1.6]{LM2022}.

Second, the coproducts $\Delta(\bGP_\nu)$ and $\Delta(\bGQ_\nu)$ are always 
finite linear combinations of tensor products of the form $\bGP_\lambda\otimes \bGP_\mu$ and $\bGQ_\lambda\otimes \bGQ_\mu$.
By \eqref{Delta-eq} and \eqref{ab-eq} this is equivalent to the numbers $\widehat a_{\lambda\mu}^\nu$ and $\widehat b_{\lambda\mu}^\nu$
being nonzero for only finitely many pairs $(\lambda,\mu)$ when $\nu$ is fixed.
This holds since both numbers are zero if $\lambda \not\subseteq \nu$ (by definition)
or if $\mu \not\subseteq\nu$ (since $\widehat a_{\lambda\mu}^\nu=\widehat a_{\mu\lambda}^\nu$ and $\widehat b_{\lambda\mu}^\nu=\widehat b_{\mu\lambda}^\nu$ as $\mSym$ is cocommutative).

To show that $\mathbf{GP} $ and $\mathbf{GQ}$  are not Hopf algebras,
it suffices to check that $\bJP_\lambda =\pm \antipode(\bGP_\lambda)$ and $\bJQ_\lambda = \pm \antipode(\bGQ_\lambda)$ 
may fail to be finite linear combinations of 
$\bGP$- and $\bGQ$-functions. 
This can already be seen for $\lambda =(1)$ by setting $x_i=0$ for all $i>1$.
Under this specialization one has $\bJP_{(1)} = \frac{-x_1}{1+\beta x_1}$,
$\bJQ_{(1)} = \frac{-2x_1+\beta x_1^2}{1+\beta x_1}$,
$\bGP_{(n)} = x_1^n$, and $\bGQ_{(n)} = (2  + \beta x_1)x_1^n$ for all $n \in \PP$,
while $\bGP_{\mu} = \bGQ_{\mu} =0$ whenever $\ell(\mu)>1$, so the relevant expansions are clearly infinite.
\end{proof}

\begin{remark}
The span of the stable Grothendieck polynomials $\{G^{(\beta)}_\lambda\}$ is a bialgebra by \cite[Cor.~6.7]{Buch2002}.
A similar argument shows that this bialgebra is also not a Hopf algebra,
as $G^{(\beta)}_{(1)}= \bGP_{(1)} $.
\end{remark}

%
%
%
%

 It is often of interest to derive cancellation-free antipode formulas.
 We should point out that the results in this section are mostly not of this form,
 as we do not know how to expand
  $\bJP_{\lambda/\mu}$, $\bJQ_{\lambda/\mu}$, $\bjp_{\lambda/\mu}$, and $\bjq_{\lambda/\mu}$
 in the respective $\{ \bGP_\nu\}$, $\{ \bGQ_\nu\}$, $\{ \bgp_\nu\}$, and $\{ \bgq_\nu\}$ bases.

We have also not discussed
the multi-Malvenuto-Reutenauer Hopf algebras $\mMPR$ and $\MMPR$.
The problem of finding cancellation-free antipode formulas for these Hopf algebras appears to be open.
Progress on this question would give $K$-theoretic generalizations of the results in \cite[\S5]{AguiarSottile}.

\subsection{Positivity properties}\label{open-sect}

To conclude this article, we collect some open problems and conjectures related to positivity properties of our various symmetric functions.
%
%
Let $\mG$ and $\mg$ denote the respective (finite) $\NN[\beta]$-linear spans of 
the stable Grothendieck polynomials $\{G^{(\beta)}_\lambda\}$
and their dual versions $\{g^{(\beta)}_\lambda\}$, with $\lambda$ ranging over all partitions.
Buch \cite[Cors. 5.5 and 6.7]{Buch2002} derives Littlewood-Richardson rules 
for (co)products of stable Grothendieck polynomials, which
imply that $G^{(\beta)}_\lambda G^{(\beta)}_\mu \in \mG$ and $\Delta(G^{(\beta)}_{\lambda}) \in   \mG\otimes \mG$
for all partitions $\lambda$ and $\mu$.

Similarly, let $\mGP$, $\mGQ$, $\mgp$, and $\mgq$ be the respective (finite) $\NN[\beta]$-linear spans of $\{\bGP_\lambda\}$,
$\{\bGQ_\lambda\}$, $\{\bgp_\lambda\}$, 
and $\{\bgq_\lambda\}$, with $\lambda$ ranging over all strict partitions.
It is known that $\bGP_\lambda\bGP_\mu \in \mGP$ and $\bGQ_\lambda\bGQ_\mu \in \mGQ$
for all strict partitions $\lambda$ and  $\mu$, or equivalently that the integers $a_{\lambda\mu}^\nu$ and $b_{\lambda\mu}^\nu$ in \eqref{ab-eq}
are always nonnegative \cite[Thm.~1.6]{LM2022}.
By Proposition~\ref{abcd-prop}, this implies that we always have $\bgp_{\lambda/\mu} \in   \mgp$ and $\bgq_{\lambda/\mu} \in   \mgq$.


Computations support some other conjectural positivity properties:

\begin{conjecture}\label{co-gp-conj}
One has $\bGP_{\lambda\ss \mu} \in \mGP$ and $\bgq_{\lambda}\bgq_\mu   \in \mgq$
for all strict partitions $\lambda$, $\mu$.
\end{conjecture}


\begin{conjecture}\label{co-gq-conj}
One has $\bGQ_{\lambda\ss \mu} \in \mGQ$ and $\bgp_{\lambda}\bgp_\mu   \in \mgp$
for all strict partitions $\lambda$, $\mu$.
\end{conjecture}

These conjectures are equivalent 
 to the inequalities $\widehat b_{\lambda\mu}^\nu\geq0$ and $\widehat a_{\lambda\mu}^\nu\geq0$,
or
via \eqref{Delta-eq} to the coproduct identities
$\Delta(\bGP_\lambda) \in \mGP\otimes\mGP$ and $\Delta(\bGQ_\lambda) \in \mGQ\otimes\mGQ$.
We do not know how to leverage the geometric interpretation of $\bGP_\lambda$ and $\bGQ_\lambda$ in \cite[\S8.3]{IkedaNaruse} to prove these properties.

 Littlewood-Richardson rules are known
 for the coefficients $a_{\lambda\mu}^\nu$  ; see \cite[Thm.~1.2]{CTY} or \cite[\S8]{PechenikYong}.
Outside some special cases considered in \cite{BuchRavikumar,LM2022}, the following problem is open: 
 
\begin{problem} Find combinatorial interpretations of the integers 
$b_{\lambda\mu}^\nu$, $\widehat a_{\lambda\mu}^\nu$, and $\widehat b_{\lambda\mu}^\nu$ in \eqref{ab-eq}.
\end{problem}

 \begin{remark*}
As noted in \cite[Conj. 5.15]{LM2019}, it also seems to hold that
$\bGP_{\lambda/ \mu} \in \mGP$
and $\bGQ_{\lambda/ \mu} \in \mGQ$.
By \eqref{ss-eq}, these containments would imply Conjectures~\ref{co-gp-conj} and \ref{co-gq-conj}.
The analogous property for 
stable Grothendieck polynomials $G^{(\beta)}_{\lambda/\mu}$ indexed by skew shapes 
follows from \cite[Thm.~6.9]{Buch2002}.
 \end{remark*}

The \defn{conjugate (dual) stable Grothendieck polynomials} are given by
\[
J^{(\beta)}_\lambda := (-1)^{|\lambda|} \antipode\(G^{(\beta)}_{\lambda}\) = (-1)^{|\lambda|} \omega\(G^{(-\beta)}_{\lambda}\)
\quand
j^{(\beta)}_\lambda := (-1)^{|\lambda|} \antipode\(g^{(\beta)}_{\lambda}\) =  (-1)^{|\lambda|} \omega\(g^{(-\beta)}_{\lambda}\)
\] for partitions $\lambda$. The 
second equalities in these definitions hold by \cite[Thm.~4.6]{Yeliussizov2019}.
Setting $\beta=-1$ turns  $J^{(\beta)}_\lambda$ into the \defn{weak set-valued tableau} generating function $J_\lambda$
in \cite[\S9.7]{LamPyl}.
Setting $\beta=1$ turns $j^{(\beta)}_\lambda$
 into the \defn{valued-set tableau} generating function $j_\lambda$
in \cite[\S9.8]{LamPyl}.
It follows from \cite[\S9]{LamPyl} that
\be\label{Jj-eq} J^{(\beta)}_\lambda = (-\beta)^{-|\lambda|} J_\lambda(-\beta x_1,-\beta x_2,\dots)
\quand
j^{(\beta)}_\lambda = \beta^{|\lambda|} j_\lambda(\beta^{-1} x_1,\beta^{-1} x_2,\dots).
\ee
The power series $\{J^{(\beta)}_\lambda\}$ and $\{j^{(\beta)}_\lambda\}$
are another pair of dual bases for $\mSym$ and $\MSym$ relative to the form $\langle\cdot,\cdot\rangle$,
since this inner product is  $\antipode$-invariant.

Below, we use the term \defn{Schur positive} to refer to any element of $\mSym$ that can be expressed as a
possibly infinite linear combination of Schur functions with coefficients in $\NN[\beta]$.

\begin{theorem}[\cite{LamPyl,Lenart2000}]
\label{both-thm}
For each partition $\lambda$, both
$G^{(\beta)}_\lambda$ and $j^{(\beta)}_\lambda$ are Schur positive, while
$s_\lambda$ is both a finite $\NN[\beta]$-linear combination of $g^{(\beta)}_\mu$'s
and an infinite $\NN[\beta]$-linear combination of $J^{(\beta)}_\mu$'s.
\end{theorem}

\begin{proof}
A few algebraic manipulations are needed to derive this statement from \cite{LamPyl,Lenart2000}. 
First, \cite[Thm.~2.8]{Lenart2000}  expresses $s_\lambda$ as an infinite $\NN$-linear combination of $G^{(-1)}_\mu$ functions.
On substituting $x_i \mapsto \beta x_i$, dividing both sides by $\beta^{|\lambda|}$, and applying $\omega$,
this becomes a $\NN[\beta]$-linear expansion of $s_\lambda$ into $J^{(\beta)}_\mu$ functions.
By duality $j^{(\beta)}_\lambda$ is Schur positive; in view of \eqref{Jj-eq}, this also follows from \cite[Thm.~9.8]{LamPyl},
which gives the Schur expansion of $g_\lambda = \omega(j_\lambda)$.
Finally, \cite[Thm.~2.2]{Lenart2000} gives a positive combinatorial interpretation of the Schur expansion
of $G^{(\beta)}_\lambda$, 
and by duality we have $s_\lambda \in \mg$.
\end{proof}

It is known that $\bGP_{\lambda} $ and $\bGQ_{\lambda} $ are both in $ \mG$ and hence Schur positive,
for any strict partition $\lambda$ \cite[Thms.~3.27 and 3.40]{MarScr2023}.
%
%
%
Combining \cite[Cor.~4.7]{MP2020} and \cite[Thm.~4.17]{MP2021} with the results in \cite{BKSTY}
gives an algorithm to compute
the $G^{(\beta)}_\mu$ terms appearing in $\bGP_{\lambda}$.
The only known algorithm to do the same for $\bGQ_\lambda$
is to expand the right side of \eqref{GQ-to-GP-eq}, which may involve cancellations.

Computations 
suggest some other instances of Schur positivity:


\begin{conjecture}
\label{bgq-mq-conj2}
If $\lambda$ is a strict partition then $\bjp_\lambda$ and  $\bjq_\lambda$ are Schur positive.
\end{conjecture}

The more interesting open problem implicit in this conjecture
is the following:

\begin{problem}
Find combinatorial interpretations of the coefficients in the expansions
of $\bGP_{\lambda}$, $\bGQ_\lambda$,
$\bjp_\lambda$, and  $\bjq_\lambda$
into Schur functions
and
stable Grothendieck polynomials.
\end{problem}

The \defn{canonical (dual) stable Grothendieck functions} $G^{(\alpha,\beta)}_\lambda$
and $g^{(\alpha,\beta)}_\lambda$ 
are generalizations of stable Grothendieck polynomials introduced in \cite{Yeliussizov2017}.
In our notation, they satisfy $G^{(0,\beta)}_\lambda=G^{(\beta)}_\lambda$, $G^{(-\beta,0)}_\lambda=J^{(\beta)}_\lambda$,
$g^{(0,-\beta)}_\lambda=g^{(\beta)}_\lambda$, and $g^{(\beta,0)}_\lambda=j^{(\beta)}_\lambda$.
Both $G^{(\alpha,\beta)}_\lambda$
and $g^{(\alpha,\beta)}_\lambda$ are Schur positive   by \cite[Thm.~4.6]{HawScr} and \cite[Thm.~9.8]{Yeliussizov2017}.
This suggests another   open problem:

\begin{problem}
Describe the shifted analogues $\GP^{(\alpha,\beta)}_\lambda$ and
 $\GQ^{(\alpha,\beta)}_\lambda$
 (respectively, $\gp^{(\alpha,\beta)}_\lambda$ and $\gq^{(\alpha,\beta)}_\lambda$) of the power series $G^{(\alpha,\beta)}_\lambda$
(respectively, $g^{(\alpha,\beta)}_\lambda$) and  prove similar positivity results.
\end{problem}

Theorem~\ref{both-thm} implies that 
 $G^{(\beta)}_\lambda$ is an infinite $\NN[\beta]$-linear combination of $J^{(\beta)}_\mu$'s
 and that  $j^{(\beta)}_\lambda$ is a finite $\NN[\beta]$-linear combination of $g^{(\beta)}_\mu$'s.
 Patrias gives an explicit description of the coefficients in these expansions in \cite[Thm.~59]{Patrias}
 using the notation $G_\lambda :=  G^{(-1)}_\lambda $ and $\tilde j_\lambda := j^{(-1)}_\lambda $.
 
There appears to be a shifted analogue of this result.
 Here, we write $\mJP$ and $\mJQ$ for the respective sets of infinite $\NN[\beta]$-linear combinations of $\bJP$- and $\bJQ$-functions.

\begin{conjecture}\label{last1-conj}
If $\lambda$ is a strict partition then $\bGP_\lambda\in \mJP$
and $\bjq_\lambda \in \mgq$.
\end{conjecture}

\begin{conjecture}\label{last2-conj}
If $\lambda$ is a strict partition then $\bGQ_\lambda\in \mJQ$
and $\bjp_\lambda \in \mgp$.
\end{conjecture}

As usual, beyond simply proving these conjectures, the following is of interest:

\begin{problem}
Find combinatorial interpretations of the coefficients appearing in the positive expansions suggested by Conjectures~\ref{last1-conj} and \ref{last2-conj}
\end{problem}

\printbibliography

\end{document}